\documentclass[12pt,reqno]{amsart}
\usepackage[new]{old-arrows}
\usepackage{amssymb}
\usepackage{amsmath}
\usepackage[all]{xy}

\newtheorem{theorem}{Theorem}[section]
\newtheorem{proposition}[theorem]{Proposition}
\newtheorem{lemma}[theorem]{Lemma}
\newtheorem{corollary}[theorem]{Corollary}
\newtheorem*{question}{Question}

\theoremstyle{definition}
\newtheorem{definition}[theorem]{Definition}
\newtheorem{remark}[theorem]{Remark}

\DeclareMathOperator{\End}{End}
\DeclareMathOperator{\Hom}{Hom}
\DeclareMathOperator{\Tr}{Tr}
\DeclareMathOperator{\Tw}{Tw}
\DeclareMathOperator{\Pic}{Pic}
\DeclareMathOperator{\rk}{rk}
\DeclareMathOperator{\Supp}{Supp}
\DeclareMathOperator{\charpoly}{char}
\DeclareMathOperator{\Sec}{Sec}

\newcommand{\bigslant}[2]{{\raisebox{.2em}{$#1$}\left/\raisebox{-.2em}{$#2$}\right.}}
\newcommand{\hooklongrightarrow}{\lhook\joinrel\longrightarrow}

\numberwithin{equation}{section}

\addtocontents{toc}{\protect\setcounter{tocdepth}{1}}

\begin{document}

\title[Vector bundles over hyperk\"ahler twistor spaces]{On vector bundles
over hyperk\"ahler twistor spaces}

\author[I. Biswas]{Indranil Biswas}

\address{School of Mathematics, Tata Institute of Fundamental
Research, Homi Bhabha Road, Mumbai 400005, India}

\email{indranil@math.tifr.res.in}

\author[A. Tomberg]{Artour Tomberg}

\address{Department of Mathematics, Western University, Middlesex College, London, Ontario, Canada, N6A 5B7}
\address{Also: Faculty of Mathematics, National Research University Higher School of Economics, 6 Usacheva St., 
Moscow, Russia, 119048}

\email{artour@tomberg.com}

\subjclass[2010]{32L25, 53C28, 32L10}

\keywords{Hyperk\"ahler manifold, twistor space, stability, twistor line, holomorphic connection.}

\date{}

\begin{abstract}
We study the holomorphic vector bundles $E$ over the twistor space $\Tw(M)$ of a compact simply
connected hyperk\"ahler manifold $M$. We give a characterization of the semistability condition
for $E$ in terms of its restrictions to the holomorphic sections of the holomorphic twistor projection
$\pi \,:\, \Tw(M)\,\longrightarrow\, \mathbb{CP}^1$. It is shown that if $E$ admits a holomorphic
connection, then $E$ is holomorphically trivial and the holomorphic connection on $E$ is trivial
as well. For any irreducible vector
bundle $E$ on $\Tw(M)$ of prime rank, we prove that its restriction to the generic fibre
of $\pi$ is stable. On the other hand, for a K3 surface $M$, we construct examples of irreducible vector
bundles of any composite rank on $\Tw(M)$ whose restriction to every fibre of $\pi$ is non-stable.
We have obtained a new method of constructing irreducible vector bundles on hyperk\"ahler 
twistor spaces; this method is employed in constructing these examples.
\end{abstract}

\maketitle

\tableofcontents

\section{Introduction}

Twistor theory was introduced by Penrose \cite{penrose} in the 1960s as a way of relating physical fields on 
Minkowski space-time to complex analytic objects on the projective space $\mathbb{CP}^3$. The idea has proved
to be rich not only for theoretical physics, but for mathematics as well, in particular for the study of the
geometry of 4-manifolds. Given an oriented Riemannian 4-manifold $X$, one can associate to it its twistor
space $Z$, which is a 6-manifold with an almost complex structure. This almost complex structure on $Z$
is integrable precisely if the Riemannian metric on $X$ is self-dual \cite{AHS}. In the situation
where the almost complex structure on $Z$ is integrable, a version of the twistor correspondence
relates the conformal geometry of $X$ with the complex analytic geometry of $Z$.

The twistor theory of oriented Riemannian 4-manifolds can be generalized to the hyperk\"ahler setting, where it 
takes the following form. We first recall that a Riemannian manifold $M$ is called hyperk\"ahler if it has a triple of 
integrable almost complex structures $I,\, J,\, K$, which are parallel with respect to the Levi-Civita connection 
of $M$ and satisfy the quaternionic relations $I^2 \,=\, J^2 \,=\, K^2 \,=\, -\text{Id}$, $IJ \,=\, -JI \,=\, K$. 
Such a manifold $M$ is endowed with a 2-sphere $S^2\,=\, \{(a,b,c)\, \in\, {\mathbb R}^3\,\mid\, a^2+b^2+c^2\,=\,1\}$
of induced complex structures, given by linear combinations $aI + bJ + cK$, where $(a,\, b,\, c)\, \in\, S^2$.
In this situation, the twistor space of $M$, which we denote by $\Tw(M)$, is a Hermitian manifold which is
canonically diffeomorphic (but not biholomorphic) to the Cartesian
product $M\times {{}S}^2$, and thus has natural projections
\begin{equation}\label{twist-proj}
\xymatrix{& \Tw(M) \ar[rd]^\pi \ar[dl]_\sigma & \\ M & & \mathbb{CP}^1,}
\end{equation}
the second of which is a holomorphic map, while the fibres of $\sigma$ are complex submanifolds identified with
$\mathbb{CP}^1$. It should be clarified that $\sigma$ is \textit{not} holomorphic.

In this setting, the holomorphic structure of the complex manifold $\Tw(M)$ completely encodes the quaternionic 
structure of the hyperk\"ahler manifold $M$. For example, $M$ can be recovered from $\Tw(M)$ (see Theorem 1 in 
\cite{hitchin2}). More generally, there is a version of the twistor correspondence (see Theorem 5.12 in 
\cite{kaled-verbit}) which associates to every vector bundle on $M$, with a connection whose curvature has type 
$(1,\, 1)$ with respect to any of the complex structures parametrized by $S^2$, a holomorphic vector bundle on 
$\Tw(M)$ whose restrictions to all the fibres of the projection $\sigma$ in the diagram \eqref{twist-proj} are 
trivial, and vice versa. This bijective correspondence leads to an identification of the corresponding moduli 
spaces, and in this way studying vector bundles on the twistor space $\Tw(M)$ can help in our understanding of the 
geometry of the original manifold $M$.

In the present paper, we pursue this idea further and study holomorphic vector bundles $E$ on the twistor space 
$\Tw(M)$ of a compact simply connected hyperk\"ahler manifold $M$, in particular, we
investigate the relationship between their 
stability and the stability of their restrictions to the fibres of the projections $\sigma$ and $\pi$ in the 
diagram \eqref{twist-proj}.

The fibres of $\sigma\, :\, \Tw(M) \,\longrightarrow\, M$ in $\Tw(M)$ are called horizontal twistor lines; more 
generally, holomorphic sections of $\pi$ are called twistor lines in $\Tw(M)$. We show that the semistability of 
$E$ on $\Tw(M)$ can be related to the ``semistability'' of its restrictions to the twistor lines in $\Tw(M)$. More 
precisely, we show that if $E$ restricts semi-stably to the image of some twistor line $s\, :\, \mathbb{CP}^1\, 
\longrightarrow\, \Tw(M)$, then $E$ itself is semistable. On the other hand, if $E$ is semistable, then for some 
twistor line $s\, :\, \mathbb{CP}^1\, \longrightarrow\, \Tw(M)$, either the restriction $s^*E$ is semistable, or 
the slopes of the associated graded components of the Harder--Narasimhan filtration
$$
0\,=\, {\mathcal E}_0\, \subsetneq\, {\mathcal E}_1\, \subsetneq\, {\mathcal E}_2\, \subsetneq\,
\cdots\, \subsetneq\, {\mathcal E}_{n-2}\, \subsetneq\, {\mathcal E}_{n-1}\, \subsetneq\,
{\mathcal E}_n\,=\, s^*E
$$
of $s^*E$ satisfy the condition
$\mu({\mathcal E}_i/{\mathcal E}_{i-1}) \,=\, \mu({\mathcal E}_{i+1}/{\mathcal E}_{i}) +1$ for all $1\, \leq\, i\,\leq\, n-1$ (Theorem \ref{thm1}). We also show that if $E$ is a holomorphic bundle on $\Tw(M)$ admitting a holomorphic connection $D$, then $E$ is holomorphically trivial, and $D$ is the trivial connection (Proposition \ref{prop1}).

Concerning the restrictions of a holomorphic bundle $E$ on $\Tw(M)$ to the fibres of the twistor projection $\pi\, 
:\, \Tw(M) \,\longrightarrow\, \mathbb{CP}^1$, a result of Kaledin and Verbitsky
shows that if $E$ restricts stably to the generic fibre of $\pi$, then it is an irreducible 
bundle on $\Tw(M)$, in the sense that it does not have any nonzero proper subsheaf of lower rank
(see Lemma 7.3 in \cite{kaled-verbit}). In the paper \cite{tomberg3}, the second author proved a
partial converse to this result (see Theorem \ref{thm:result} in the present article, which is Theorem 4.1 in \cite{tomberg3}), while the following question was posed about the full converse:

\begin{question}[{\cite[p.~2]{tomberg3}}]
Given an irreducible bundle $E$ on the twistor space $\Tw(M)$, will it always be stable on the generic
fibre of the twistor projection $\pi\,:\, \Tw(M) \,\longrightarrow\, \mathbb{CP}^1$?
\end{question}

Note that if we replace ``irreducible'' by ``stable'', the answer is negative: in \cite{tomberg2}, an example of a stable (but not irreducible) bundle on $\Tw(M)$ with non-stable restrictions 
to the fibres of $\pi$ was constructed.

In the present article, we prove that an irreducible $E$ of prime rank on $\Tw(M)$ does restrict stably to the generic fibre of $\pi\,:\, \Tw(M) \,\longrightarrow\, \mathbb{CP}^1$ (Theorem \ref{thm:prime}). However, we also show that for $M$ a K3 surface, there are examples of bundles $E$ on $\Tw(M)$ of any composite
rank which are irreducible but whose restrictions to all the fibres of $\pi$ are non-stable (see Theorem 
\ref{thm:composite}). This settles the question above in the negative, and also strengthens the result of 
\cite{tomberg2}.

The proof of Theorem \ref{thm:prime} gives a new method of constructing irreducible bundles on the twistor space $\Tw(M)$. The significance of this comes from the fact that irreducible bundles (which only exist on nonalgebraic manifolds) are notoriously difficult to study, and the main difficulty is in fact a lack of general mechanisms of constructing such bundles. There are only specific methods for particular classes of manifolds; for the case of surfaces, see \cite{toma, ABT,teleman-toma,brin-mor}. Our proof gives a new method of constructing irreducible bundles on a 3-dimensional complex manifold, namely the twistor space $\Tw(M)$ of a K3 surface $M$.

\section{Preliminaries} \label{sect:preliminaries}

We begin by giving the definitions of the basic objects that we shall be working with, and also recalling the
results that will be used in the subsequent sections.

\begin{definition}\label{hyperk}
A \emph{hyperk\"ahler} structure on a smooth manifold $M$ consists of a triple of integrable almost
complex structures $I,\, J,\, K\, :\, TM \,\longrightarrow\, TM$ satisfying
\[
I^2 \,=\, J^2 \,=\, K^2 \,=\, -\text{Id},\ \ IJ \,=\, -JI \,=\, K,
\]
together with a Riemannian metric $g$ on $M$ which is simultaneously K\"ahler with respect to $I,\, J,\, K$.
\end{definition}

Together with the identity mapping, $I,\, J,\, K$ induce an action of the quaternion algebra
$\mathbb{H}$ on the smooth tangent bundle $TM$, which is moreover parallel with respect to the Levi-Civita connection
on $M$
associated to the Riemannian metric $g$. Any linear combination $A = aI + bJ + cK$, where $(a,\, b,\, c)\, \in\, {\mathbb R}^3$ with
$a^2 + b^2 + c^2 \,=\, 1$, is an endomorphism of the tangent bundle $TM$ satisfying $A^2 \,=\,
-\text{Id}_{TM}$, and thus defines an almost complex structure on $M$. This almost complex structure $A$
is actually integrable and the metric $g$ is again K\"ahler with respect to this
complex structure. In this way, we get a family of \emph{induced complex structures}
\[
{{}S}^2 \,=\, \left\{aI + bJ + cK\, \mid\, a^2 + b^2 + c^2 \,=\, 1 \right\}
\]
on $M$ parametrized by $S^2$. Consider the (topological) product manifold $\Tw(M) := M \times {{}S}^2$. For every
point $$(m,\, A) \,\in\, \Tw(M) \,=\, M \times {{}S}^2\, ,$$ we have the tangent space decomposition
$$T_{(m,A)} \Tw(M)\,=\, T_m M \bigoplus T_A {{}S}^2\, .$$ Identifying ${{}S}^2$ with $\mathbb{CP}^1$ using the
stereographic projection from $(0,\, 0,\, 1)$, we have the almost complex structure
$I_{{{}S}^2} \,: \,T {{}S}^2
\,\longrightarrow\, T {{}S}^2$ on ${{}S}^2$, while any $A\,=\, (a,\, b,\, c)\, \in\, {{}S}^2$ itself defines
the almost complex structure $A\,=\, aI + bJ + cK\, :\, T M \,\longrightarrow\,
T M$ on $M$ mentioned earlier. The operator $$\mathcal{I} \,:\, T\Tw(M) \,\longrightarrow\,
T\Tw(M)\, ,$$ which at the point $(m,\, A)$ is the direct sum $A(m)\bigoplus I_{{{}S}^2}(A)$, satisfies
the equation ${{}\mathcal{I}}^2 \,=\, -\text{Id}_{T\Tw(M)}$, and thus defines an almost complex structure on $\Tw(M)$. It can be
shown that $\mathcal{I}$ is actually integrable \cite{sal}.

\begin{definition}\label{twistor}
The above complex manifold $(\Tw(M),\, \mathcal{I})$ is called the \emph{twistor space} of the hyperk\"ahler manifold $M$.
\end{definition}

Thinking of ${{}S}^2 \cong \mathbb{CP}^1$ as the set of induced complex structures of $M$ as above, the twistor space $\Tw(M)$ parametrizes these structures at points of $M$. We have the canonical projections
\begin{equation}\label{dpi}
\xymatrix{& \Tw(M) \ar[rd]^\pi \ar[dl]_\sigma & \\ M & & \mathbb{CP}^1.}
\end{equation}
With the complex structure of $\Tw(M)$ described above, it is easy to verify that the second projection $\pi \,:\, \Tw(M) \,\longrightarrow\,
\mathbb{CP}^1$ is a holomorphic map. Holomorphic sections of $\pi$ will be called the \emph{twistor lines}, while
the constant sections of the form
$$I \,\longmapsto\, (m,\, I) \,\in\, M \times \mathbb{CP}^1 \,= \,\Tw(M)\, ,$$ where $m \in M$, will be
called the horizontal twistor lines. The hyperk\"ahler metric $g$ on $M$ and the Fubini-Study metric
$g_{\mathbb{CP}^1}$ on $\mathbb{CP}^1$ together produce
a natural Hermitian metric
\[
\sigma^*(g) + \pi^*\left(g_{\mathbb{CP}^1}\right)
\]
on $\Tw(M)$.
The Hermitian metric on $\Tw(M)$ thus obtained is not K\"ahler but satisfies the weaker property of being 
\emph{balanced} (see \cite{kaled-verbit}), i.e., its Hermitian form $\omega$ satisfies
the equation $d\left(\omega^{n-1}\right) = 
0$, where $n$ is the complex dimension of $\Tw(M)$ (which is clearly $\frac{\dim_{\mathbb R}M}{2}+1$).

From now on, the original complex structures $I,\, J,\, K$ will not play any vital role. We will denote an 
arbitrary induced complex structure on $M$ by $I$, and the resulting complex manifold by $M_I$. As noted above, $g$ 
is a K\"ahler metric on $M_I$. These $M_I$ are precisely the fibres of the holomorphic twistor projection $\pi \,:\, 
\mathrm{Tw}(M) \,\longrightarrow\, \mathbb{CP}^1$, and it will be useful to think of $\Tw(M)$ as the collection of 
K\"ahler manifolds $M_I$ lying above the points $I \,\in\, \mathbb{CP}^1$ via the map $\pi$. From now on, we shall 
assume throughout the article that $M$ is compact.

Recall that a hyperk\"ahler manifold $M$ has an action of the quaternion algebra $\mathbb{H}$ on its tangent 
bundle, which is parallel with respect to the Levi-Civita connection for its hyperk\"ahler metric $g$. Restricting to the group of 
unitary quaternions in $\mathbb{H}$, we get an action of ${\rm SU}(2)$ on $TM$, hence also on the bundle of 
differential forms ${{}\Omega}_M^*$. Since the action is parallel, it commutes with the Laplace operator, and thus 
preserves harmonic forms. Applying Hodge theory, we get a natural action of ${\rm SU}(2)$ on the cohomology 
${{}H}^*(M,\,\mathbb{C})$.

\begin{lemma}\label{thm:invariant}
A differential form $\eta$ on a hyperk\"ahler manifold $M$ is ${\rm SU}(2)$-invariant if and only if it is of Hodge
type $(p,\,p)$ with respect to all induced complex structures $M_I$.
\end{lemma}

\begin{proof}
This is proved in Proposition 1.2 of \cite{verbit1}.
\end{proof}

Using Lemma \ref{thm:invariant}, we can define vector bundles on $M$ which are simultaneously holomorphic with 
respect to all induced complex structures $I\,\in\, \mathbb{CP}^1$.

\begin{definition}\label{def:hyperholo}
Let $M$ be a hyperk\"ahler manifold and $B$ a $C^\infty$ vector bundle on $M$. A connection $\nabla$ on $B$ is 
called \emph{hyperholomorphic} if it preserves a Hermitian metric on $B$ and its curvature $\mathcal{K}(\nabla) 
\,\in\, H^0\left(M, \,\End(B) \bigotimes \Omega^2_M\right)$ is ${\rm SU}(2)$-invariant.
\end{definition}

By Lemma \ref{thm:invariant}, the ${\rm SU}(2)$-invariance
condition is equivalent to $\mathcal{K}(\nabla)$ being a $C^\infty$ section of
\[
\End(B) \otimes \left(\bigcap_{I \in \mathbb{CP}^1} {{}\Omega}^{1,1}_{M_I}\right) \,\subseteq\,
\End(B) \otimes {{}\Omega}^2_M.
\]
This means that for any $I\,\in \,\mathbb{CP}^1$, the $(0,\,1)$-part $\nabla^{0,1}_I$ of $\nabla$ with respect to $I$ 
induces a holomorphic structure on $B$ over $M_I$ \cite[p.~9, Proposition 3.7]{Ko}; we shall denote the 
corresponding holomorphic bundle by $E_I$. In this way, a hyperholomorphic connection $\nabla$ gives a family of 
holomorphic vector bundles $E_I$ over the K\"ahler manifolds $M_I$, $I\,\in\, \mathbb{CP}^1$, all with the same
underlying $C^\infty$ vector bundle $B$. To assemble these bundles into one object, we can use the twistor formalism.

Recall that the twistor space $\Tw(M)$ comes equipped with a (nonholomorphic) projection $\sigma \,: 
\,\Tw(M)\,\longrightarrow\, M$. Given a hyperholomorphic bundle $(B,\, \nabla)$ on $M$, consider the pullback 
bundle with connection $(\sigma^*\!B,\, \sigma^*\nabla)$ on $\Tw(M)$. By the considerations in the previous 
paragraph and the structure of $\Tw(M)$, the curvature of the connection $\sigma^*\nabla$ on $\Tw(M)$ is of type 
$(1,\,1)$, hence its $(0,\,1)$-part $\left(\sigma^*\nabla\right)^{0,1}$ defines a holomorphic structure on the 
topological bundle $\sigma^*\!B$ over $\Tw(M)$, which we shall denote by $E$. The restriction of $E$ to the fibre 
$\pi^{-1}(I) = M_I$ of the holomorphic projection $\pi \,:\, \Tw(M) \,\longrightarrow\, \mathbb{CP}^1$ is none 
other than the holomorphic bundle $E_I$ described in the previous paragraph. Here we will use the term 
``hyperholomorphic bundle'' interchangeably to refer either to a $C^\infty$ bundle with connection $(B,\, \nabla)$ 
on $M$ as in the statement of Definition \ref{def:hyperholo}, or to the holomorphic bundle $E_I$ on $M_I$ obtained 
from it as described in the previous paragraph, or to the holomorphic bundle $E$ on $\Tw(M)$ constructed above. In 
either of these contexts, the hyperholomorphic line bundles form a complex abelian Lie group under tensor product.

\begin{definition}\label{def:generic}
Let $M$ be hyperk\"ahler and $I$ an induced complex structure. We say that $I$ is \emph{generic} with respect to the hyperk\"ahler structure on $M$ if all elements in
\[
\bigoplus_p {{}H}^{p,p}(M_I) \cap {{}H}^{2p}(M,\,\mathbb{Z}) \subset {{}H}^*(M,\,\mathbb{C})
\]
are ${\rm SU}(2)$-invariant.
\end{definition}

This terminology is justified: most induced complex structures are generic, in a sense made precise in the following proposition.

\begin{proposition}\label{thm:dense}
Let $M$ be a hyperk\"ahler manifold. The subset $S_0 \,\subset\, {{}S}^2$ of generic induced complex structures is
dense in ${{}S}^2$ and its complement ${{}S}^2\setminus S_0$ is countable.
\end{proposition}

\begin{proof}
This is proved in Proposition 2.2 of \cite{verbit2}.
\end{proof}

We now give the definition of stable vector bundles and torsionfree sheaves. Recall that a coherent sheaf 
$\mathcal{F}$ on an arbitrary complex manifold is called \emph{torsionfree} if the natural morphism into the double 
dual $\mathcal{F} \,\longrightarrow\, \mathcal{F}^{**}$ is injective. If it is an isomorphism, $\mathcal{F}$ is 
called \emph{reflexive}. We call the sheaf $\mathcal{F}$ \emph{normal} if for every open set $U$ and every analytic 
subset $A \subset U$ of complex codimension at least two, the restriction map $\mathcal{F}(U) \,\longrightarrow\, 
\mathcal{F}(U \setminus A)$ is an isomorphism.

\begin{definition}
Let $Z$ be a compact balanced manifold of complex dimension $n$, and let $\omega$ denote the Hermitian
form of its balanced metric. The \emph{degree} of a coherent sheaf $\mathcal{F}$ on $Z$ is defined to be
\[
\deg \mathcal{F} \,:=\, \int_Z c_1(\mathcal{F}) \wedge \omega^{n-1}\, ,
\]
where by $c_1(\mathcal{F})$ we mean any representative of the first Chern class of $\mathcal{F}$ in the
de Rham cohomology ${{}H}^2(Z,\, \mathbb{C})$ (the condition that $\omega$ is balanced ensures
that $\deg \mathcal{F}$ does not depend on the choice of the representative of the first Chern class).
If $\mathcal{F}$ is a nonzero torsionfree coherent sheaf, the \emph{slope} of $\mathcal{F}$ is
\[
\mu(\mathcal{F}) \,:=\, \frac{\deg \mathcal{F}}{\rk \mathcal{F}}\,.
\]
A torsionfree sheaf $\mathcal{F}$ is called \emph{stable} (respectively, \emph{semistable}) if for every subsheaf $\mathcal{G} \subset \mathcal{F}$ with $0 < \rk \mathcal{G} < \rk \mathcal{F}$ we have
\[
\mu(\mathcal{G}) \,< \, \mu(\mathcal{F}) \ \ \left(\textrm{respectively, } \mu(\mathcal{G}) \,\le\,
\mu(\mathcal{F})\right)\, ,
\]
while $\mathcal{F}$ is called \emph{polystable} if it is a direct sum of stable sheaves of the same slope.

A torsionfree sheaf $\mathcal{F}$ is called \emph{irreducible} if it has no proper subsheaves of lower rank.
\end{definition}

Note that any irreducible sheaf is stable.

For a hyperk\"ahler $M$, as mentioned previously, its twistor space $\Tw(M)$ is a balanced manifold, and the fibres $\pi^{-1}(I) = M_I$ of the twistor projection $\pi \,:\, \Tw(M) \,\longrightarrow\,
 \mathbb{CP}^1$ are K\"ahler. Moreover, if we denote by $\omega$ the Hermitian form of the balanced metric on $\Tw(M)$, its restriction $\omega_I$ to $M_I$ is precisely the K\"ahler form of the K\"ahler metric on $M_I$. Thus, given a holomorphic vector bundle $E$ on $\Tw(M)$, it makes sense to talk both about its stability as a bundle on $\Tw(M)$, and also the stability of its restrictions $E_I$ to the fibres $M_I$ of $\pi$.

\begin{definition}
A holomorphic vector bundle $E$ on the twistor space $\Tw(M)$ is called \emph{fibrewise stable} if its restriction 
$E_I$ to each fibre $M_I$ of the projection $\pi \,:\, \Tw(M) \,\longrightarrow\, \mathbb{CP}^1$ is stable. $E$ is 
called \emph{generically fibrewise stable} if $E_I$ is stable for all $I$ in a nonempty Zariski open subset of 
$\mathbb{CP}^1$. Similarly, $E$ is called \emph{fibrewise simple} if all the restrictions $E_I$ are simple bundles, 
in the sense that $\Hom_{M_I}(E_I,\, E_I) \,=\, \mathbb{C}$, and it is called \emph{generically fibrewise simple} if
$E_I$ is simple for all $I$ in a nonempty Zariski open subset of $\mathbb{CP}^1$.
\end{definition}

The following important result gives a topological characterization of bundles on $M$ admitting hyperholomorphic 
structures.

\begin{theorem}\label{thm:c1c2}
Let $M$ be a hyperk\"ahler manifold, and let $E_I$ be a stable holomorphic bundle on $M_I$, for some $I \,\in\, 
\mathbb{CP}^1$. If $c_1(E_I)$ and $c_2(E_I)$ are ${\rm SU}(2)$-invariant, then there exists a unique hyperholomorphic 
connection on the underlying $C^\infty$ bundle of $E_I$ which induces the holomorphic structure of $E_I$. 
\end{theorem}

\begin{proof}
This is proved in Theorem 2.5 of \cite{verbit1}.
\end{proof}

Any hyperholomorphic bundle on any $M_I$ has degree zero, as shown by the following lemma.

\begin{lemma}\label{thm:degr}
An ${\rm SU}(2)$-invariant 2-form $\beta$ on a hyperk\"ahler manifold $M$ satisfies
\[
\int_M \beta \wedge \omega_I^{n-1} = 0
\]
for any induced complex structure $I$, where $\omega_I$ denotes the K\"ahler form on $M_I$.
\end{lemma}

\begin{proof}
This is a consequence of Lemma 2.1 of \cite{verbit1}.
\end{proof}

It follows from Lemma \ref{thm:degr} that any sheaf $\mathcal{F}$ on any $M_I$ with
${\rm SU}(2)$-invariant first 
Chern class $c_1(\mathcal{F}) \,\in\, {{}H}^2(M,\, \mathbb{R})$ has degree zero, because the harmonic representative of 
$c_1(\mathcal{F})$ must be ${\rm SU}(2)$-invariant as a two-form. In particular, if $S_0\,\subseteq\, 
{{}S}^2\,\cong\, \mathbb{CP}^1$ denotes the subset of generic complex structures of $M$ as in the statement of Proposition 
\ref{thm:dense}, then for any $I \,\in\, S_0$, all sheaves on $M_I$ have degree zero, and are thus semistable.
The following proposition is a consequence of this.

\begin{proposition}\label{thm:divisors}
The holomorphic twistor projection $\pi \,:\, \Tw(M)\,\longrightarrow\,
\mathbb{CP}^1$ establishes a bijective correspondence between divisors on $\mathbb{CP}^1$ and those on $\Tw(M)$.
\end{proposition}

\begin{proof}
This is proved in Proposition 2.17 of \cite{tomberg3}.
\end{proof}

In view of this bijective correspondence, given a divisor $D$ on $\mathbb{CP}^1$, we will denote
by the same letter $D$ the corresponding divisor on $\Tw(M)$, and vice versa. The corresponding
line bundle on $\mathbb{CP}^1$ will be denoted by $\mathcal{O}_{\mathbb{CP}^1}(D)$, and on
$\Tw(M)$ by $\mathcal{O}_{\Tw(M)}(D)$. For a sheaf $\mathcal{F}$ on $\Tw(M)$, we define
\[
\mathcal{F}(D) \,:=\, \mathcal{F} \bigotimes\nolimits_{\mathcal{O}_{\Tw(M)}} \mathcal{O}_{\Tw(M)}(D)\, .
\] 

We finish this section by stating a theorem from the paper \cite{tomberg3} on fibrewise 
stability of bundles on the twistor space of a simply connected hyperk\"ahler manifold $M$, 
which establishes a partial converse to the result of Kaledin and Verbitsky
proved in \cite{kaled-verbit} that was mentioned in the introduction.

\begin{theorem}\label{thm:result}
Let $M$ be a compact simply connected hyperk\"ahler manifold, and let $E$ be a holomorphic
vector bundle on the twistor space $\Tw(M)$. If $E$ is generically fibrewise stable, then it is
irreducible. If $E$ is irreducible of rank two or three, then $E$ is
generically fibrewise stable. Also, if $E$ is irreducible and it is generically fibrewise simple,
then $E$ is generically fibrewise stable.
\end{theorem}

\begin{proof}
The forward implication follows from the proof of Lemma 7.3 in \cite{kaled-verbit}. The two partial
converses are proved in Theorem 4.1 of \cite{tomberg3} and also Theorem 4.2.1 of \cite{tomberg1}.
\end{proof}

In Section \ref{sect:prime}, we shall strengthen this result by showing that the converse holds 
for arbitrary vector bundles of prime rank. On the other hand, 
there are examples of irreducible bundles of any composite rank on 
$\Tw(M)$, for $M$ a K3 surface, which are not generically fibrewise stable (this
is proved in Section \ref{sect:composite}).

\section{Semistability of bundles and holomorphic connections on 
$\Tw(M)$}\label{sect:hol-conn-semistab}

\subsection{Semistability and restriction to twistor lines}

Let $M$ be a compact hyperk\"ahler manifold, and $\Tw(M)$ its twistor space. Let
${\mathcal D}$ denote the component of the Douady space for $\Tw(M)$ that contains
the horizontal twistor lines. Let
\begin{equation}\label{e3}
\Sec(\pi)\, \subset\, {\mathcal D}
\end{equation}
be the Zariski open subset consisting of the holomorphic sections of the projection
$\pi\, :\, \Tw(M)\, \longrightarrow\, {\mathbb C}{\mathbb P}^1$, that is, twistor lines in $\Tw(M)$.

The following theorem does not require $M$ to be simply connected.

\begin{theorem}\label{thm1}
Let $E$ be a torsionfree coherent analytic sheaf on $\Tw(M)$.

If for some element $s\, \in\, \Sec(\pi)$, the pulled back coherent analytic sheaf
$s^*E\, \longrightarrow\, {\mathbb C}{\mathbb P}^1$ is torsionfree and semistable, then
$E$ is semistable.

If $E$ is semistable, then one of the following two holds:
\begin{itemize}
\item There is a nonempty Zariski open subset ${\mathcal U}_s\, \subset\, \Sec(\pi)$ such that
for all $s\, \in\, {\mathcal U}_s$, the pulled back coherent analytic sheaf
$s^*E\, \longrightarrow\, {\mathbb C}{\mathbb P}^1$ is
torsionfree and semistable.

\item For all element $s\, \in\, \Sec(\pi)$ such that the pulled back coherent analytic sheaf
$s^*E\, \longrightarrow\, {\mathbb C}{\mathbb P}^1$ is torsionfree, the vector bundle
$s^*E$ is not semistable. Furthermore, there is a nonempty Zariski open subset
${\mathcal U}^0_s\, \subset\, \Sec(\pi)$ such that
for all $s\, \in\, {\mathcal U}^0_s$, the pulled back coherent analytic sheaf
$s^*E\, \longrightarrow\, {\mathbb C}{\mathbb P}^1$ is
torsionfree, and if
$$
0\,=\, {\mathcal E}_0\, \subsetneq\, {\mathcal E}_1\, \subsetneq\, {\mathcal E}_2\, \subsetneq\,
\cdots\, \subsetneq\, {\mathcal E}_{n-2}\, \subsetneq\, {\mathcal E}_{n-1}\, \subsetneq\,
{\mathcal E}_n\,=\, s^*E
$$
is the Harder--Narasimhan filtration of $s^*E$ (which is not semistable by the first
sentence), then
$$
\mu({\mathcal E}_i/{\mathcal E}_{i-1}) \,=\, \mu({\mathcal E}_{i+1}/{\mathcal E}_{i}) +1
$$
for all $1\, \leq\, i\,\leq\, n-1$.
\end{itemize}
\end{theorem}

\begin{proof}
Since $\Tw(M)$ is topologically the product of $M$ and $\mathbb{CP}^1$, we have
\begin{equation}\label{a1}
{{}H}^2(\Tw(M),\, \mathbb{R})\,=\, {{}H}^2(M,\, \mathbb{R})\oplus
{{}H}^2({\mathbb C}{\mathbb P}^1,\, \mathbb{R})\, .
\end{equation}
For any torsionfree coherent analytic
sheaf $\mathcal F$ on $\Tw(M)$, it is clear that all its restrictions $\mathcal F_I \,=\,
\mathcal F\vert_{\pi^{-1}(I)}$ for all $I\, \in\, {\mathbb C}{\mathbb P}^1$ have the same
first Chern class, and hence the harmonic representative of
$c_1(\mathcal F_I) \,\in\, {{}H}^2(M,\, \mathbb{Z})$ is ${\rm SU(2)}$-invariant as
a consequence of Lemma \ref{thm:invariant}. From Lemma \ref{thm:degr}, and the
discussion following it, we know that the degree of $\mathcal{F}_I$ is zero for
all $I\, \in\, {\mathbb C}{\mathbb P}^1$. Consequently, using \eqref{a1} it follows that
\begin{equation}\label{e4}
\deg({\mathcal F})\, =\, \deg(s^*{\mathcal F})\cdot \text{Vol}(M)\, ,
\end{equation}
where $s\, :\, {\mathbb C}{\mathbb P}^1\, \longrightarrow\, \Tw(M)$ is any element of
$\Sec(\pi)$ defined in \eqref{e3}.

{}From \eqref{e4} it follows immediately that $E$ is semistable if for the general element 
$s\, \in\, \Sec(\pi)$, the pulled back coherent analytic sheaf $s^*E\, 
\longrightarrow\, {\mathbb C}{\mathbb P}^1$ is torsionfree and semistable. Now the 
openness of the semistability condition (see \cite[p.~635, Theorem~2.8(B)]{Ma} for it), implies
that if $s^*_0E$ is torsionfree and semistable for some $s_0\, 
\in\, \Sec(\pi)$, then $s^*E$ is torsionfree and semistable for the general element 
$s\, \in\, \Sec(\pi)$. Therefore, $E$ is semistable if $s^*E$ is torsionfree and 
semistable for some $s\, \in\, \Sec(\pi)$.

Now assume that $E$ is semistable. Consider all $s\, \in\, \Sec(\pi)$
such that the pulled back coherent analytic sheaf
$s^*E\, \longrightarrow\, {\mathbb C}{\mathbb P}^1$ is torsionfree. Their
locus is a nonempty Zariski open subset of $\Sec(\pi)$. This Zariski open
subset of $\Sec(\pi)$ will be denoted by ${\mathcal D}^1$.

Assume that for some $s_0\,\in\, {\mathcal D}^1$, the torsionfree sheaf $s^*_0E$ is semistable. 
Then from the openness of semistability condition it follows that there is a nonempty Zariski 
open subset ${\mathcal U}_s\, \subset\, {\mathcal D}^1$ such that for all $s\, \in\, {\mathcal 
U}_s$, the pulled back coherent analytic sheaf $s^*E\, \longrightarrow\, {\mathbb C}{\mathbb 
P}^1$ is torsionfree and semistable.

Therefore, assume that $s^*_0E$ is not semistable for every $s_0\,\in\, {\mathcal D}^1$.
Consequently, there is a nonempty Zariski open subset
$$
{\mathcal U}'_s\, \subset\, {\mathcal D}^1
$$
such that for every $s_0\, \in\, {\mathcal U}'_s$, the Harder--Narasimhan
filtration of $s_0^*E$ is independent of $s_0$. In other words,
the Harder--Narasimhan filtrations of all $s_0^*E$, $s_0\, \in\, {\mathcal U}'_s$,
have equal length, and the rank and degree of the $i$--th term in the filtration
are independent of $s_0\, \in\, {\mathcal U}'_s$ for all $i$.

Take an element $s\, \in\, {\mathcal U}'_s\, .$ Let
\begin{equation}\label{hn2}
0\,=\, {\mathcal E}_0\, \subsetneq\, {\mathcal E}_1\, \subsetneq\, {\mathcal E}_2\, \subsetneq\,
\cdots\, \subsetneq\, {\mathcal E}_{n-2}\, \subsetneq\, {\mathcal E}_{n-1}\, \subsetneq\,
{\mathcal E}_n\,=\, s^*E
\end{equation}
be the Harder--Narasimhan filtration of $s^*E$.

Let $\mathcal N$ be the normal bundle of
$s({\mathbb C}{\mathbb P}^1)\, \subset\,\Tw(M)$. We know that $${\mathcal N}\,=\,
{\mathcal O}_{{\mathbb C}{\mathbb P}^1}(1)^{\oplus 2d}\, ,$$ where $4d$ is the real dimension
of $M$ (see \cite[p.~142, Theorem 1(2)]{hitchin2}). From this it follows
that the evaluation homomorphism
$$
\varepsilon\, :\, {\mathbb C}{\mathbb P}^1\times {{}H}^0({\mathbb C}{\mathbb P}^1,\,
{\mathcal N})\, \longrightarrow\, {\mathcal N}\, ,
$$
that sends any $(x,\, v)\, \in\,{\mathbb C}{\mathbb P}^1\times {{}H}^0({\mathbb C}{\mathbb P}^1,\,
{\mathcal N})$ to $v(x)\, \in\, {\mathcal N}_x$, is surjective, and moreover
\begin{equation}\label{dt}
\text{kernel}(\varepsilon)\,=\, 
{\mathcal O}_{{\mathbb C}{\mathbb P}^1}(-1)^{\oplus 2d}\, .
\end{equation}

For every $1\, \leq\, i\, \, \leq\, n-1$, there is a natural homomorphism
\begin{equation}\label{dh}
\Psi^s_i\, :\, \text{kernel}(\varepsilon)\otimes {\mathcal E}_i\, \longrightarrow\,
(s^*E)/{\mathcal E}_i
\end{equation}
(see \eqref{hn2}); these homomorphisms $\Psi^s_i$ correspond to the infinitesimal deformation
of the subsheaves of the Harder--Narasimhan filtrations.

From the given condition that $E$ is semistable it can be deduced that
the homomorphism $\Psi^s_i$ is not identically zero for the general element
$s\, \in\, {\mathcal U}'_s$. Indeed, if $\Psi^s_i\,=\, 0$ for all
$s\, \in\,{\mathcal U}'_s$, then there is a subsheaf
$${\mathbb E}\, \subset\, E$$ such that ${\mathcal E}_i\,=\,s^*{\mathbb E}$
for all $s\, \in\, {\mathcal U}'_s$. Now using \eqref{e4}, and the properties
of the Harder--Narasimhan filtration, it follows that
the subsheaf $\mathbb E$ of $E$ contradicts the semistability condition for $E$.

Since the homomorphism $\Psi^s_i$ in \eqref{dh} is nonzero for the general
element $s\, \in\, {\mathcal U}'_s$, using \eqref{dt} it follows that
\begin{equation}\label{eh}
\mu({\mathcal E}_i/{\mathcal E}_{i-1}) \,=\, \mu({\mathcal E}_{i+1}/{\mathcal E}_{i}) +1
\end{equation}
for all $s\, \in\, {\mathcal U}'_s$ such that $\Psi^s_i\, \not=\, 0$. In other words,
there is a nonzero Zariski open subset ${\mathcal U}^0_s\, \subset\, {\mathcal U}'_s$
such that \eqref{eh} holds. This completes the proof of the theorem.
\end{proof}

A simply connected compact hyperk\"ahler manifold is called irreducible if it
is not a product of hyperk\"ahler manifolds.

\begin{proposition}\label{prop3}
Let $M$ be a compact simply connected irreducible 
hyperk\"ahler manifold. Then the holomorphic
tangent bundle ${{}T}^{1,0}\Tw(M)$ of the twistor space $\Tw(M)$ is stable.
\end{proposition}

\begin{proof}
Consider the projection $\pi$ in \eqref{dpi}. We have the short exact sequence of
holomorphic vector bundles on $\Tw(M)$
\begin{equation}\label{z1}
0\, \longrightarrow\, {{}T}^{1,0}_{\pi}
\, \longrightarrow\, {{}T}^{1,0}{\Tw(M)} \, \stackrel{d\pi}{\longrightarrow}\,
\pi^*{{}T}^{1,0}{{\mathbb C}{\mathbb P}^1} \, \longrightarrow\, 0\, ,
\end{equation}
where ${{}T}^{1,0}_{\pi}$ is the relative holomorphic tangent bundle for the
projection $\pi$, and $d\pi$ is the differential of $\pi$.

Firstly, the relative tangent bundle ${{}T}^{1,0}_{\pi}$ is irreducible. Indeed, for any $I\, \in\, {\mathbb C}{\mathbb P}^1$, the restriction ${{}T}^{1,0}_{\pi}\vert_{\pi^{-1}(I)}$ is stable, because $M$ is simply connected and irreducible; note that $M_I \,=\, \pi^{-1}(I)$ admits a K\"ahler--Einstein metric \cite{yau} (Calabi's conjecture). Hence the vector bundle ${{}T}^{1,0}_{\pi}$ is fibrewise stable, and by the forward implication of Theorem \ref{thm:result}, we conclude that it is irreducible.

Assume that ${{}T}^{1,0}{\Tw(M)}$ is not stable. Let
$$
F\, \subsetneq\, {{}T}^{1,0}{\Tw(M)}
$$
be a nonzero subsheaf such that $({{}T}^{1,0}{\Tw(M)})/F$ is torsionfree and
$$
\mu(F)\, \geq\, \mu({{}T}^{1,0}{\Tw(M)})\, .
$$
Since ${{}T}^{1,0}_{\pi}$ is irreducible, from \eqref{z1} we now conclude the following:
\begin{enumerate}
\item either $F\,=\, {{}T}^{1,0}_{\pi}$, or

\item $F$ is a subsheaf of ${{}T}^{1,0}{\Tw(M)}$ of rank one such that the composition
$$
F\, \hooklongrightarrow\, {{}T}^{1,0}\Tw(M)
\, \stackrel{d\pi}{\longrightarrow}\, 
\pi^*{{}T}^{1,0}{{\mathbb C}{\mathbb P}^1}
$$
is not identically zero.
\end{enumerate}

Firstly observe that $\mu({{}T}^{1,0}_{\pi})\, <\, \mu({{}T}^{1,0}{\Tw(M)})$,
because the slope of the restriction of ${{}T}^{1,0}_{\pi}$ to a horizontal twistor
line is strictly less than the slope of the restriction of ${{}T}^{1,0}\Tw(M)$ to a horizontal twistor
line. Therefore, ${{}T}^{1,0}_{\pi}$ does not
destabilize ${{}T}^{1,0}{\Tw(M)}$.

Secondly, it can be shown that there is no rank one subsheaf
$$F'\, \subset\, {{}T}^{1,0}{\Tw(M)}$$ such that the composition
$$
F\, \hooklongrightarrow\, {{}T}^{1,0}\Tw(M)\, \stackrel{d\pi}{\longrightarrow}\, 
\pi^*{{}T}^{1,0}{{\mathbb C}{\mathbb P}^1}
$$
is not identically zero. To prove this, restrict the exact sequence in \eqref{z1}
to $M_I = \pi^{-1}(I)$. This produces the short exact sequence
$$
0\, \longrightarrow\, {{}T}^{1,0} M_I
\, \longrightarrow\, ({{}T}^{1,0}{\Tw(M)})\vert_{M_I} \, \stackrel{d\pi}{\longrightarrow}\,
M_I\times T^{1,0}_I{\mathbb C}{\mathbb P}^1 \, \longrightarrow\, 0\, ,
$$
where $M_I\times T^{1,0}_I{\mathbb C}{\mathbb P}^1 \, \longrightarrow\,
M_I$ is the trivial holomorphic line bundle with fibre $T^{1,0}_I{\mathbb C}{\mathbb P}^1$.
This short exact sequence does not split holomorphically. Indeed, the obstruction class
to its splitting, which lies in $\text{Hom}(T_I{\mathbb C}{\mathbb P}^1,\,
{{}H}^1(M_I,\, TM_I))$, is the Kodaira--Spencer homomorphism for the family $\Tw(M)$.
Consequently, there is no
rank one subsheaf
$$F'\, \subset\, {{}T}^{1,0}{\Tw(M)}$$ such that the composition
$$
F\, \hooklongrightarrow\, {{}T}^{1,0}\Tw(M)\, \stackrel{d\pi}{\longrightarrow}\, 
\pi^*{{}T}^{1,0}{{\mathbb C}{\mathbb P}^1}
$$
is not identically zero. This completes the proof.
\end{proof}

\begin{remark}
Note that the restriction of $T^{1,0}\Tw(M)$ to a twistor line decomposes as
${\mathcal O}_{{\mathbb C}{\mathbb P}^1}(2)\bigoplus
{\mathcal O}_{{\mathbb C}{\mathbb P}^1}(1)^{\oplus 2d}$. Also, \cite{tomberg2} gives
an example of a stable rank 2 bundle on $\Tw(M)$ for $M$ a K3 surface whose restriction
to a twistor line is not semistable.
Therefore, the converse of the first part of Theorem \ref{thm1} does not hold.
\end{remark}

\subsection{Holomorphic connections}

Let $E$ be a holomorphic vector bundle on a complex manifold $Z$. A \textit{holomorphic connection}
on $E$ is a holomorphic differential operator
$$
D\, :\, E\, \longrightarrow\, E\otimes\Omega^{1,0}_Z
$$
of order one satisfying the Leibniz identity which says that
$$
D(fs)\,=\, f\cdot D(s)+ s\otimes df\, ,
$$
where $s$ is any locally defined holomorphic section of $E$ and $f$ is any locally defined 
holomorphic function on $Z$ \cite{At}. Let $\overline{\partial}_E\, :\, E\, \longrightarrow\, 
E\bigotimes\Omega^{0,1}_Z$ be the Dolbeault operator defining the holomorphic structure on $E$. 
Then for any holomorphic connection $D$ on $E$, the differential operator 
$D+\overline{\partial}_E$ is a usual connection on the holomorphic vector bundle $E$. Since the 
differential operator $D$ is holomorphic, the curvature $(D+\overline{\partial}_E)^2$ of the 
connection $D+\overline{\partial}_E$ is a holomorphic section of 
$\text{End}(E)\bigotimes\Omega^{2,0}_Z$.

As before, let $M$ be a compact hyperk\"ahler manifold and $\Tw(M)$ the corresponding
twistor space. For the following proposition we do not assume that $M$ is simply connected.

\begin{proposition}\label{prop1}
Let $E$ be a holomorphic vector bundle on $\Tw(M)$ equipped with a holomorphic
connection $D$. Then the curvature of $D$ vanishes identically.
\end{proposition}

\begin{proof}
Let
$$
{\mathcal K}(D)\,:=\, (D+\overline{\partial}_E)^2\, \in\, {{}H}^0(\Tw(M),\, \text{End}(E)\otimes\Omega^{2,0}_{\Tw(M)})
$$
be the curvature of the connection $D+\overline{\partial}_E$. To show that
${\mathcal K}(D)$ vanishes identically,
let
$$s \, :\, {\mathbb C}{\mathbb P}^1\, \longhookrightarrow\, \Tw(M)$$
be a horizontal twistor line. The holomorphic
vector bundle $s^* \Omega^{1,0}_{\Tw(M)}\, \longrightarrow\, {\mathbb C}{\mathbb P}^1$ will be denoted by
$\mathbb V$. We note that
\begin{equation}\label{e1}
s^*{\mathcal K}(D)\, \in\, {{}H}^0({\mathbb C}{\mathbb P}^1,\, \text{End}(s^* E)\otimes
\bigwedge\nolimits^2 \mathbb V)\,=\, {{}H}^0({\mathbb C}{\mathbb P}^1,\, (s^* \text{End}(E))\otimes
\bigwedge\nolimits^2 \mathbb V)\, ;
\end{equation}
to clarify, $s^*{\mathcal K}(D)$ is the pullback of the section ${\mathcal K}(D)$
and not the restriction of the differential form.

Now, $s^*D$ is a holomorphic connection on the holomorphic vector bundle $s^*E\, \longrightarrow\,
{\mathbb C}{\mathbb P}^1$. But any holomorphic connection on a Riemann surface $Y$ is flat (curvature vanishes
identically) because $\Omega^{2,0}_Y\,=\, 0$. Therefore, $(s^*E,\, s^*D)$ is given by a representation
of the fundamental group. Since ${\mathbb C}{\mathbb P}^1$ is simply connected, we conclude that
the holomorphic vector bundle $s^*E$ is holomorphically trivial. Fix a holomorphic trivialization
$${\mathcal O}_{\mathbb{CP}^1}^{\oplus r}\, \stackrel{\sim}{\longrightarrow}\, s^*E\, ,$$
where $r\, =\,\text{rank}(E)$. Using this trivialization,
$s^*{\mathcal K}(D)$ in \eqref{e1} is a holomorphic section
\begin{equation}\label{e2}
s^*{\mathcal K}(D)\, \in\, {{}H}^0({\mathbb C}{\mathbb P}^1,\, \bigwedge\nolimits^2 {\mathbb V})^{\oplus r^2}\, .
\end{equation}

As before, let $\mathcal N$ denote the normal bundle of $s({\mathbb C}{\mathbb P}^1)\, \subset\,\Tw(M)$. Recall that ${\mathcal N}\,=\, {\mathcal O}_{{\mathbb C}{\mathbb P}^1}(1)^{\oplus 2d}$, where
$4d$ is the real dimension of $M$. From this it follows immediately that
$$
{\mathbb V}^*\, =\,{\mathcal O}_{{\mathbb C}{\mathbb P}^1}(1)^{\oplus 2d}\oplus
{\mathcal O}_{{\mathbb C}{\mathbb P}^1}(2)\, .
$$
Therefore, we have
$$
{{}H}^0({\mathbb C}{\mathbb P}^1,\, \bigwedge\nolimits^2 {\mathbb V})\,=\, 0\, .
$$
Hence from \eqref{e2} it follows that $s^*{\mathcal K}(D)\,=\, 0$. This implies that
the curvature ${\mathcal K}(D)$ vanishes identically.
\end{proof}

\begin{corollary}\label{cor1}
Let $M$ be a simply connected compact hyperk\"ahler manifold. Let $(E,\, D)$ be a holomorphic bundle,
on the corresponding twistor space $\Tw(M)$, equipped with a holomorphic connection.
Then the vector bundle $E$ is holomorphically trivial and $D$ is the trivial connection on it.
\end{corollary}

\begin{proof}
Since $\Tw(M)$ is simply connected, this follows from Proposition \ref{prop1}.
\end{proof}

\section{Finite base extensions of the twistor projection} \label{sect:twmx}

Let $M$ be a compact simply connected hyperk\"ahler manifold, $\Tw(M)$ its twistor space and $\pi \,:\, \Tw(M)\,
\longrightarrow\,\mathbb{CP}^1$ the natural holomorphic twistor projection. In this section, we will examine
the fibre product
\begin{equation} \label{diagramma}
\xymatrix{\Tw(M)_X \ar[r]^-{\varphi} \ar[d]_{\pi_X} & \Tw(M) \ar[d]^\pi \\ X \ar[r]_f & \mathbb{CP}^1,}
\end{equation}
where $f \,:\, X \,\longrightarrow\, \mathbb{CP}^1$ is an arbitrary holomorphic branched cover of 
$\mathbb{CP}^1$ by a smooth projective curve $X$. Observe that we have $\pi_* \mathcal{O}_{\Tw(M)} \,=\, 
\mathcal{O}_{\mathbb{CP}^1}$, since the fibres $M_I$ of the map $\pi$ are connected, and for a
similar reason there is an isomorphism
 ${\pi_X}_* \mathcal{O}_{\Tw(M)_X} \,=\, \mathcal{O}_{X}$. Also, the maps $\pi$ and 
$\pi_X$ both induce embeddings of the corresponding Picard groups since they both admit
holomorphic sections.

We would like to relate the Picard group of $\Tw(M)_X$ to the Picard group of $\Tw(M)$. In general, Picard groups 
of fibred products cannot be described in a nice way in terms of the Picard groups of the factors, but in our 
particular case we do have such a description. We first describe $\Pic \Tw(M)$.

\begin{proposition}\label{thm:picard}
The following isomorphism
\[
\Pic \Tw(M) \cong \Pic \mathbb{CP}^1 \oplus {{}H}^2(M,\, \mathbb{Z})_{\mathrm{inv}}
\]
holds, where $${{}H}^2(M,\, \mathbb{Z})_{\mathrm{inv}}\,\subseteq\, {{}H}^2(M,\, \mathbb{Z})$$
is the subgroup of ${\rm SU}(2)$-invariant cohomology classes. More precisely, $\Pic \Tw(M)$ is
the direct sum of its subgroup $\pi^*\left(\Pic \mathbb{CP}^1\right)$ and the subgroup of
hyperholomorphic line bundles on $\Tw(M)$.
\end{proposition}

\begin{proof}
First, we will show that $\Pic \Tw(M)$ is discrete.

To prove this, note that since $M$ is simply connected, we have
$${{}H}^1(M,\, \mathbb{C}) \,=\, 0\, ,$$ and applying Hodge theory, for any induced complex structure
$I \,\in\, \mathbb{CP}^1$, we conclude that $${{}H}^{0,1}(M_I)\, =\, {{}H}^1(M_I, \,\mathcal{O}_{M_I})
\,=\, 0\, .$$
By Grauert's theorem (Theorem 10.5.5 in \cite{grauert-remmert}), it follows from this that ${{}R}^1\pi_*\mathcal{O}_{\Tw(M)} = 0$ and, as mentioned above, $\pi_* \mathcal{O}_{\Tw(M)} = \mathcal{O}_{\mathbb{CP}^1}$. So we have
\[
{{}H}^0(\mathbb{CP}^1,\, {{}R}^1\pi_*\mathcal{O}_{\Tw(M)}) \,=\,
{{}H}^1(\mathbb{CP}^1,\, \pi_*\mathcal{O}_{\Tw(M)}) \,=\, 0\, .
\]
Examining the Leray spectral sequence of the twistor projection $\pi \,:\, \Tw(M)
\,\longrightarrow\,
 \mathbb{CP}^1$ for the sheaf $\mathcal{O}_{\Tw(M)}$, we see that
\begin{equation}\label{cz}
{{}H}^1(\Tw(M), \, \mathcal{O}_{\Tw(M)}) \,= \,0\, .
\end{equation}
Next consider the exponential sequence of sheaves on $\Tw(M)$
$$
0\, \longrightarrow\, 2\pi\sqrt{-1}{\mathbb Z} \, \longrightarrow\, {\mathcal O}_{\Tw(M)}
\, \stackrel{\exp}{\longrightarrow}\, {\mathcal O}^*_{\Tw(M)} \, \longrightarrow\, 0\, .
$$
Let
$$
H^1(\Tw(M),\, {\mathcal O}_{\Tw(M)}) \, \longrightarrow\, \Pic \Tw(M)\,=\, H^1(\Tw(M),\, {\mathcal O}^*_{\Tw(M)})
$$
$$
\, \longrightarrow\, H^2(\Tw(M),\,2\pi\sqrt{-1}{\mathbb Z})
$$
be the corresponding long exact sequence of cohomologies. Using \eqref{cz} in this exact sequence we conclude
that the above homomorphism
$$
\Pic \Tw(M)\,=\, H^1(\Tw(M),\, {\mathcal O}^*_{\Tw(M)})
\, \longrightarrow\, H^2(\Tw(M),\,2\pi\sqrt{-1}{\mathbb Z})
$$
is injective. It follows from this that $\Pic \Tw(M)$ is discrete. More
precisely, the holomorphic structure of a holomorphic line bundle on $\Tw(M)$ is completely determined
by its underlying topological structure.

Since $\Tw(M)$ is topologically the product of $\mathbb{CP}^1$ and $M$, we have
\begin{equation} \label{razlozhenie}
{{}H}^2(\Tw(M),\, \mathbb{Z}) \,=\, {{}H}^2(\mathbb{CP}^1,\, \mathbb{Z})\oplus {{}H}^2(M,
\,\mathbb{Z})\, .
\end{equation}
As noted above, the group homomorphism $\pi^* \,:\, \Pic \mathbb{CP}^1 \,\longrightarrow\,
\Pic \Tw(M)$ is injective, and so we can think of $\Pic \mathbb{CP}^1 \,\cong\, \mathbb{Z}$ as a
subgroup of $\Pic \Tw(M)$; it corresponds to the first summand ${{}H}^2(\mathbb{CP}^1,\,
\mathbb{Z})$ in \eqref{razlozhenie}. On the other hand, it follows from Lemma
\ref{thm:invariant}, Theorem \ref{thm:c1c2} and the simple connectedness of $M$
that the group of hyperholomorphic line bundles on $M$ is isomorphic to ${{}H}^2(M,\,
\mathbb{Z})_{\mathrm{inv}}$. The corresponding hyperholomorphic line bundles on $\Tw(M)$ can
thus be identified with the subgroup ${{}H}^2(M,\, \mathbb{Z})_{\mathrm{inv}} \,\subseteq\,
{{}H}^2(M,\, \mathbb{Z})$ of the second summand in \eqref{razlozhenie}.

It only remains
to observe that for an arbitrary holomorphic line bundle $L$ on $\Tw(M)$, the second part
of $c_1(L)$ according to the decomposition \eqref{razlozhenie} lies in
${{}H}^2(M, \,\mathbb{Z})_{\mathrm{inv}}$. Indeed, observe that the restrictions $L_I$ of $L$
to the fibres $M_I \,=\, \pi^{-1}(I)$ of the projection $\pi \,:\, \Tw(M) \,
\longrightarrow\, \mathbb{CP}^1$ are all isomorphic topologically, and $c_1(L_I) \,\in\,
{{}H}^{1,1}(M_I) \cap {{}H}^2(M, \,\mathbb{Z})$, so $c_1(L_I)$ must be ${\rm SU}(2)$-invariant,
as a consequence of Lemma \ref{thm:invariant}.
\end{proof}

The maps $f$, $\pi$ in the diagram \eqref{diagramma} induce group homomorphisms from $\Pic 
\mathbb{CP}^1$ to $\Pic X$, $\Pic \Tw(M)$, respectively, and both of these are injective. Taking 
the product of these monomorphisms $\Pic \mathbb{CP}^1 \,\longrightarrow\, \Pic X \bigoplus \Pic 
\Tw(M)$, $L\, \longmapsto\, (f^*L,\, \pi^*L)$, we can thus think of $\Pic \mathbb{CP}^1$ 
as a subgroup of $\Pic X \bigoplus \Pic \Tw(M)$.

\begin{proposition}\label{thm:picard-x}
In the diagram \eqref{diagramma},
\[
\Pic \Tw(M)_X \,\cong\, \bigslant{(\Pic X \oplus \Pic \Tw(M))}{\Pic \mathbb{CP}^1}\, .
\]
\end{proposition}

\begin{proof}
There is a natural homomorphism
\begin{equation} \label{pic-iso}
\begin{array}{ccc}
\Pic X \oplus \Pic \Tw(M) & \longrightarrow & \Pic \Tw(M)_X \\
(L_1,\, L_2) & \longmapsto & \pi_X^*L_1 \otimes \varphi^*L_2^*
\end{array}
\end{equation}
We first show that the kernel is $\Pic \mathbb{CP}^1$. Suppose $(L_1,\, L_2) \,\in\,
\Pic X\bigoplus \Pic \Tw(M)$ is such that $\pi_X^*L_1 \,\cong\, \varphi^*L_2$ on $\Tw(M)_X$. This
means that the restriction of $\varphi^*L_2$ on $\Tw(M)_X$ to any fibre of the morphism $\pi_X$
is trivial, hence the same can be said about the restriction of $L_2$ on $\Tw(M)$ to
any fibre of $\pi$. By Proposition \ref{thm:picard}, the line bundle $L_2$ must be of the form
$\pi^* L'$ for some line bundle $L'$ on $\mathbb{CP}^1$. Then on $\Tw(M)_X$ we have
\[
\pi_X^*L_1 \cong \varphi^*L_2 \cong \varphi^*(\pi^* L') = \pi_X^*(f^* L').
\]
But, as noted previously, $\pi_X^* \,:\, \Pic X \,\longrightarrow\,
 \Pic \Tw(M)_X$ is injective, hence $L_1 \cong f^*L'$ on $X$. So $(L_1,\, L_2)
\,=\, (f^*L',\, \pi^*L')$ lies in the subgroup $\Pic \mathbb{CP}^1 \,\subseteq\, \Pic X \bigoplus \Pic \Tw(M)$.

We now show that the map in \eqref{pic-iso} is surjective. Let $L$ be an arbitrary holomorphic line bundle on 
$\Tw(M)_X$. Note that, for any point $P \,\in\, X$, the fibre $\pi_X^{-1}(P)$ is just the manifold $M_{f(P)}$, 
where $f(P) \in \mathbb{CP}^1$ is the corresponding induced complex structure on $M$. It follows that, when we take 
the restriction $L_P$ of $L$ to the fibre $\pi_X^{-1}(P)$, the first Chern class $c_1(L_P)\,=\, \eta\,\in\, {{}H}^2(M,\, 
\mathbb{Z})$ (which is the same for all $P$) must be an element of ${{}H}^2(M,\, \mathbb{Z})_{\mathrm{inv}}$. By 
Proposition \ref{thm:picard}, there exists a hyperholomorphic line bundle $\widetilde{L}$ on $\Tw(M)$ corresponding to 
$\eta\,\in\, {{}H}^2(M,\, \mathbb{Z})_{\mathrm{inv}}$, and taking its pullback to $\Tw(M)_X$, we have that 
$L\bigotimes \varphi^* \widetilde{L}^*$ restricts trivially to all fibres of $\pi_X$. It remains to show that the line 
bundle $L' \,:= \,L \otimes \varphi^* \widetilde{L}^*$ on $\Tw(M)_X$ comes from $X$. For any $P \,\in \,X$, we have
\[
{{}H}^0\left(\pi_X^{-1}(P), \,L'_P\right) \,=\, {{}H}^0\left(M_{f(P)},\, \mathcal{O}_{M_{f(P)}}\right) = \mathbb{C}.
\]
By Grauert's theorem (Theorem 10.5.5 in \cite{grauert-remmert}), it follows that ${\pi_X}_* L'$ is a line bundle
on $X$ and its fibre over $P \in X$ is isomorphic to the above.
Taking the pullback of ${\pi_X}_* L'$ back to $\Tw(M)_X$, we have a natural morphism of line bundles
\[
\pi_X^*({\pi_X}_* L') \longrightarrow L',
\]
and it is easy to see that it is an isomorphism over every fibre of $\pi_X$. Hence it is an isomorphism everywhere on $\Tw(M)_X$, and we are done.
\end{proof}

\section{Irreducible bundles on $\mathrm{Tw}(M)$ of prime rank} \label{sect:prime}

In this section, we shall extend Theorem \ref{thm:result} by showing that, for a compact simply connected 
hyperk\"ahler manifold $M$, any irreducible bundle of prime rank on the twistor space $\Tw(M)$ is generically 
fibrewise stable with respect to the twistor projection $\pi \,:\, \Tw(M) \,\longrightarrow\, \mathbb{CP}^1$. The 
strategy of proof consists of showing that any such bundle is generically fibrewise simple, thus reducing to the 
case already covered by Theorem \ref{thm:result}.

Let $E$ be a holomorphic vector bundle on the twistor space $\Tw(M)$. In this section we will be concerned with 
morphisms of the form $F : E \,\longrightarrow\, E(D)$, where $D$ is a divisor on $\Tw(M)$. For any divisor $D' 
\,\ge\, D$, the composition of $F$ with the natural monomorphism $E(D) \,\longrightarrow\, E(D')$ will be denoted 
by the same letter $F : E \,\longrightarrow\, E(D')$, and we will think of it as essentially the same morphism. 
Using this idea, given two morphisms $F_1 \,:\, E \,\longrightarrow\, E(D_1)$, $F_2 \,:\, E \,\longrightarrow\, 
E(D_2)$, we can think of their sum $F_1 + F_2$ as a morphism of the form $E \,\longrightarrow\, E(D')$, where $D'$ 
is any divisor containing both $D_1$ and $D_2$.

Recall from Proposition \ref{thm:divisors} that the twistor projection $\pi \,:\, \Tw(M)\,\longrightarrow\, 
\mathbb{CP}^1$ identifies divisors on $\Tw(M)$ with divisors on $\mathbb{CP}^1$, and thus in what follows we will 
denote the corresponding divisors by the same letter. In particular, the field of meromorphic functions on $\Tw(M)$ 
can be identified with $K(\mathbb{CP}^1)$, the function field of $\mathbb{CP}^1$. Thus, given an element $\eta 
\,\in\, K(\mathbb{CP}^1)$ with divisor of poles $D'$, we can think of it as a holomorphic section of the line bundle 
$\mathcal{O}_{\Tw(M)}(D')$ on $\Tw(M)$, and vice versa. In this manner, given a morphism $F\,:\, E \,\longrightarrow\, 
E(D)$ on $\Tw(M)$, we can think of the product $\eta \cdot F$ as a morphism $\eta \cdot F \,:\, E \,\longrightarrow\, 
E(D + D')$.

Now let $E$ be an irreducible bundle on $\Tw(M)$. Taking the pushforward of the endomorphism bundle $\End E\,=\, E^* 
\bigotimes E$ by the twistor projection, $\pi_*(\End E)$ is a vector bundle, being a torsionfree sheaf on 
$\mathbb{CP}^1$, and hence holomorphically decomposes as a sum of line bundles by the Birkhoff-Grothendieck
theorem. Note that $E$ is simple bundle on $\Tw(M)$, in the sense that
\[
\Hom_{\Tw(M)}(E,\, E) \,=\, \mathbb{C}\,.
\]
This is because the irreducibility of $E$ clearly implies that it is stable, and stable bundles are
simple (see Theorem 1.2.9 in Chapter 2 of \cite{okonek}). Hence we have
\[
{{}H}^0(\mathbb{CP}^1,\, \pi_*(\End E)) \,=\, {{}H}^0(\Tw(M),\, \End E)
\,=\, \Hom_{\Tw(M)}(E,\, E) \,=\, \mathbb{C}\, .
\]
It follows that in the direct sum decomposition of $\pi_*(\End E)$, there is exactly one summand of the form 
$\mathcal{O}_{\mathbb{CP}^1}$, while all other summands (if any) are negative line bundles. It's not hard to see 
that $\pi_*(\End E) \,=\, \mathcal{O}_{\mathbb{CP}^1}$ occurs precisely when $E$ is generically fibrewise simple, while 
if it is not, $\pi_*(\End E)$ also has negative summands.

As noted above, the only endomorphisms $E\,\longrightarrow\,
E$ of an irreducible $E$ on $\Tw(M)$ are homotheties, but if we look at morphisms $E\,\longrightarrow\,
E(D)$ for various divisors $D$, we get more possibilities. Using the projection formula,
\[
\Hom_{\Tw(M)}\big(E,\, E(D)\big)\,=\,{{}H}^0\big(\Tw(M),\, (\End E)(D)\big)
\,=\, {{}H}^0\big(\mathbb{CP}^1,\, \left[\pi_*(\End E)\right](D)\big)\, .
\]
With the description of $\pi_*(\End E)$ given in the previous paragraph, we see that if $E$ is generically 
fibrewise simple, the only morphisms $E \,\longrightarrow\, E(D)$ are multiplications of $\textrm{Id}_E$ by 
meromorphic functions from $K(\mathbb{CP}^1)$, while if it is not, we can always find a morphism $F \,:\, E 
\,\longrightarrow\, E(D)$ for some $D\,>\, 0$ coming from a negative summand of $\pi_*(\End E)$, which
will not be a multiplication by an element of $K(\mathbb{CP}^1)$.

\begin{definition}
For any morphism $F \,:\, E \,\longrightarrow\, E(D)$ on $\Tw(M)$, the \emph{trace} of $F$, 
denoted $\Tr F$, is a global section of the line bundle $\mathcal{O}_{\Tw(M)}(D)$ determined by the composition
\[
\mathcal{O}_{\Tw(M)} \longrightarrow \End E \longrightarrow \mathcal{O}_{\Tw(M)}(D),
\]
where the first map is induced by the identity morphism $\textrm{Id}_E \,:\, E \,\longrightarrow\,
E$, while the second map is induced by $F$. The \emph{characteristic polynomial} of $F$, denoted
$\charpoly_F$, is a polynomial over the field $K(\mathbb{CP}^1)$, which takes $t \,\in\, K(\mathbb{CP}^1)$ to
\[
\charpoly_F(t) \,= \,\det\left(t\cdot \textrm{Id}_E - F\right) \,=\, \sum_{i = 0}^r (-1)^i\Tr({{}\Lambda}^i F)
t^{r-i}\, ,
\]
where we view $\Tr({{}\Lambda}^i F) \,\in\, {{}H}^0\left(\Tw(M), \,\mathcal{O}_{\Tw(M)}(iD)\right)$
as an element of $K(\mathbb{CP}^1)$. A root $\eta \,\in\, K(\mathbb{CP}^1)$ of $\charpoly_F(t)$ (if it exists) is
called an \emph{eigenvalue} of $F$.
\end{definition}

Since $\charpoly_F$ is a polynomial with meromorphic functions as coefficients, evaluating these coefficients at 
any point $x \in \Tw(M) \setminus \Supp D$ gives a polynomial $\left.\charpoly_F\right|_x$ over the field 
$\mathbb{C}$. It's not hard to see that $\left.\charpoly_F\right|_x$ is simply the characteristic polynomial of the 
linear map $F_x \,:\, E_x \,\longrightarrow\, E(D)_x\,\cong\, E_x$. Similarly, if $\eta \,\in\, K(\mathbb{CP}^1)$ is an 
eigenvalue of $\charpoly_F(t)$ as in the above definition, then at any $x \,\in\, \Tw(M)$ outside $\Supp D$ where 
$\eta$ is defined, the evaluation $\eta(x)$ is an eigenvalue of $F_x$.

The main result of this section follows.

\begin{theorem}\label{thm:prime}
Let $M$ be a compact simply connected hyperk\"ahler manifold, and let $E$ be a holomorphic vector bundle on its 
twistor space $\Tw(M)$. If $E$ is irreducible and the rank of $E$ is prime, then $E$ is generically fibrewise 
stable.
\end{theorem}

\begin{proof}
Let $E$ be an irreducible bundle on $\Tw(M)$ and suppose $\rk E$ is a prime number. If $E$ is generically fibrewise 
simple, then an application of Theorem \ref{thm:result} gives that $E$ is generically fibrewise stable.

To prove by contradiction, we assume that $E$ is not generically fibrewise simple.

As discussed in the beginning of this section, this implies that 
there exists a divisor $D$ and a morphism $F \,:\, E \,\longrightarrow\, E(D)$, which is not a multiplication by a 
meromorphic function. Consider the characteristic polynomial of the morphism $F$,
\[
\charpoly_F(t) \,=\, t^r + c_1t^{r-1} + \ldots + c_{r-1} t + c_r\, .
\]
We can write
\begin{equation} \label{char-decompo}
\charpoly_F(t) \,=\, p_1(t)^{n_1} \cdots \ p_s(t)^{n_s}\, ,
\end{equation}
where $p_1(t),\, \cdots,\, p_s(t)$ are distinct irreducible polynomials over the function field
$K(\mathbb{CP}^1)$, and $n_1,\, \cdots,\, n_s$ are positive integers. Plugging $F$ into $\charpoly_F(t)$, we get that
\[
\charpoly_F(F) \,=\, F^r + c_1F^{r-1} + \ldots + c_{r-1} F + c_r\, .
\]
Here, the powers $F^i$ are morphisms $E \,\longrightarrow\, E(iD)$; for example, $F^2$ is the composition
\[
E \,\stackrel{F}{\longrightarrow}\, E(D) \,\stackrel{F(D)}{\longrightarrow}\, E(2D)\, ,
\]
and similarly for higher powers. Recalling from the definition of $\charpoly_F$ that the coefficients $c_{r-i}$ are 
global sections of $\mathcal{O}_{\Tw(M)}\left([r-i]D\right)$, we see that $\charpoly_F(F)$ is a well-defined 
morphism $E\,\longrightarrow\, E(rD)$. Over any point $x \in \Tw(M) \setminus \Supp D$,
\[
\left.\charpoly_F(F)\right|_x \,=\, \charpoly_{F_x}(F_x)\, ,
\]
which is zero by the Cayley-Hamilton theorem, and since $E(rD)$ has no torsion, we conclude that $\charpoly_F(F) \,:\, E \,\longrightarrow\,
 E(rD)$ is zero globally as well.

Recalling the decomposition \eqref{char-decompo}, we can write the morphism $\charpoly_F(F) \,:\,
E \,\longrightarrow\,
 E(rD)$ as the composition
\[
E \xrightarrow{p_1(F)^{n_1}} E(D_1) \xrightarrow{p_2(F)^{n_2}} E(D_2) \longrightarrow \ldots \longrightarrow E(D_{s-1}) \xrightarrow{p_2(F)^{n_s}} E(D_s),
\]
where $D_1,\, \cdots,\, D_s$ are some divisors. As noted above, this composition is zero. But since $E$ is irreducible, the only possible morphisms from $E$ to any torsionfree sheaf on $\Tw(M)$ are monomorphisms and the zero morphism, and the same can be said about the vector bundles $E(D_1), \, \cdots,\, E(D_s)$. So one of the morphisms in the above composition must be zero. Rearranging the polynomials $p_i$ if necessary, we can assume that $p_1(F)^{n_1} = 0$. Writing $p_1(F)^{n_1}$ as the composition of the morphism $p_1(F)$ with itself $n_1$ times, and repeating the exact same argument, we conclude that $p_1(F) = 0$.

We now claim that if $p_1(F) \,=\, 0$ then $p_1(t)$ is the only irreducible polynomial in the decomposition 
\eqref{char-decompo}, in other words, $s \,=\, 1$.

To prove the above claim, let $N \,\supseteq\, K(\mathbb{CP}^1)$ be a splitting field of 
$\charpoly_F(t)$. Note that
\begin{enumerate}
\item[(i)] each $p_i(t)$ splits into distinct linear factors over $N$, since it's separable as we're working in 
characteristic 0;

\item[(ii)] since $p_1(t),\, \cdots,\, p_s(t)$ are all distinct and irreducible, no two of them have a common 
linear factor over $N$.
\end{enumerate}
What this means geometrically is that for any point $x \,\in\, \Tw(M) \setminus \Supp D$ outside a divisor, each 
restriction $\left.p_1\right|_x(t),\, \cdots, \,\left.p_s\right|_x(t)$ has no repeated roots as a polynomial over 
$\mathbb{C}$, and no two of them have a common root. Since each root of each $\left.p_i\right|_x(t)$ is an 
eigenvalue of $$F_x \,:\, E_x \,\longrightarrow\, E(D)_x \,\cong\, E_x\, ,$$ the fact that
\[
\left.p_1\right|_x(F_x) \,=\, \left.p_1(F)\right|_x \,=\, 0
\]
implies that $F_x$ has no eigenvalues other than the roots of $\left.p_1\right|_x(t)$.
This means that there can be no $p_i$ other than $p_1$. This completes the proof of the claim.

Relabeling, the decomposition in \eqref{char-decompo} can be written as
\begin{equation}\label{char-decompo-new}
\charpoly_F(t) \,=\, p(t)^n\, ,
\end{equation}
where $p(t)$ is an irreducible polynomial over the field $K(\mathbb{CP}^1)$, and $n$ is a positive integer.

We now use the fact that the rank of $E$ is prime. Note that $\deg \charpoly_F(t)\, = \,\rk E$, so we have
\[
\rk E\, =\, n \cdot \deg p(t)\, .
\]
There are two cases, which we consider separately.

\subsection*{Case $n = \rk E$} In other words, $p(t)$ is a linear polynomial in \eqref{char-decompo-new} of the form
\[
p(t) \,=\, t - \eta\, ,
\]
where $\eta \in K(\mathbb{CP}^1)$ is some meromorphic function. But we know that $p(F)\, =\, 0$, so in this case
$F\, =\,\eta$. This contradicts the fact that $F$ was chosen not to be a multiplication by a meromorphic function.

\subsection*{Case $n = 1$} In other words, $\charpoly_F(t)\,=\, p(t)$ is an irreducible polynomial
in \eqref{char-decompo-new}. We pass to the splitting field $N \,\supseteq\, K(\mathbb{CP}^1)$ of
$\charpoly_F(t)$. As noted previously, over $N$ the polynomial $\charpoly_F(t)$ splits into distinct linear factors:
\begin{equation} \label{on-twmx}
\charpoly_F(t) \,=\, (t - \eta_1)(t - \eta_2)\ldots(t - \eta_r), \ \ \eta_i \in N\, .
\end{equation}
Let $f \,:\, X \,\longrightarrow\,
 \mathbb{CP}^1$ be the unique branched cover corresponding to the field extension $K(\mathbb{CP}^1) \subseteq N$, where $X$ is a smooth curve. Consider the fibred product
\begin{equation} \label{diagramma2}
\xymatrix{\Tw(M)_X \ar[r]^-{\varphi} \ar[d]_{\pi_X} & \Tw(M) \ar[d]^\pi \\ X \ar[r]_f & \mathbb{CP}^1}
\end{equation}
Pulling back $E$ and the morphism $F \,: \,E \,\longrightarrow\,
 E(D)$ by $\varphi$, we get a morphism $\varphi^*F \,:\, \varphi^*E\,\longrightarrow\,
 \varphi^*E\left(\varphi^*D\right)$ on $\Tw(M)_X$. It's not hard to see that the characteristic polynomial of $\varphi^* F$ is given by \eqref{on-twmx}, where this time we think of $\eta_1,\, \cdots,\, \eta_r$ as meromorphic functions on $\Tw(M)_X$ coming from $X$. Look at the morphism
\[
\varphi^* E \xrightarrow{\ \varphi^*F - \eta_1 \ } \varphi^* E(\widetilde{D})
\]
where $\widetilde{D}$ denotes a divisor on $\Tw(M)_X$ containing $\varphi^*D$ and the poles of $\eta_1$. Since 
$\eta_1, \, \cdots,\, \eta_r$ are all distinct, they generically have distinct values on $\Tw(M)_X$, and so for all 
$y$ in $\Tw(M)_X$ outside a divisor, the eigenspace of the linear map
\[
\left(\varphi^*F\right)_y \,: \,\left(\varphi^* E\right)_y \,\longrightarrow\,
 \left(\varphi^* E(\widetilde{D})\right)_y \,\cong\, \left(\varphi^* E\right)_y
\]
corresponding to the eigenvalue $\eta_1(y)$ has dimension one. It follows from this that the kernel of the
above morphism $\varphi^*F - \eta_1$ is a sheaf of rank one. It is clearly torsionfree and since its cokernel
is also torsionfree, it is normal (by Lemma 1.1.16 of Chapter 2 in \cite{okonek}), and hence it is a line
bundle (see Lemma 1.1.12 and Lemma 1.1.15 of Chapter 2 in \cite{okonek}). In short, we have a line subsheaf
\[
L \,\longhookrightarrow\, \varphi^*E
\]
on $\Tw(M)_X$.

We now use the results obtained in Section \ref{sect:twmx}. By Proposition \ref{thm:picard-x}, we can write
\[
L \,\cong\, \varphi^* L' \otimes \pi_X^* L''\, ,
\]
where $L'$ is some line bundle on $\Tw(M)$ and $L''$ is some line bundle on $X$. Taking the pushforward of the 
sheaf monomorphism $L\,\cong\, \varphi^* L' \bigotimes \pi_X^* L'' \,\longhookrightarrow\, \varphi^*E$ by $\varphi$, we 
have a sheaf monomorphism
\[
\varphi_*\left(\varphi^* L' \otimes \pi_X^* L''\right)\,\longhookrightarrow\, \varphi_*\left(\varphi^*E\right)
\]
on $\Tw(M)$. Applying the projection formula to the two sides, this morphism can be expressed in the following form:
\begin{equation} \label{pushforward}
L' \otimes \varphi_*\left(\pi_X^* L''\right)\,\longhookrightarrow\,
E \otimes \varphi_*\left(\mathcal{O}_{\Tw(M)_X}\right)\, .
\end{equation}
In the diagram \eqref{diagramma2} we have an isomorphism
\[
\pi^*(f_*(\mathcal{F})) \,\cong\, \varphi_*(\pi_X^*(\mathcal{F}))
\]
for any torsionfree sheaf on $X$ (see Theorem III.3.4 and its corollaries in \cite{banica}). Taking $\mathcal{F}$ 
to be $L''$ and $\mathcal{O}_X$, and using the Birkhoff-Grothendieck theorem, we can write
\[
\varphi_*(\pi_X^*L'') \,\cong\, \pi^*(f_*(L'')) \,\cong\,
\pi^*\left(\bigoplus_{l=1}^d \mathcal{O}_{\mathbb{CP}^1}(A_l)\right) \,=\,
\bigoplus_{l=1}^d \mathcal{O}_{\Tw(M)}(A_l)\, ,
\]
\[
\varphi_*(\mathcal{O}_{\Tw(M)_X}) \,= \,\varphi_*(\pi_X^*(\mathcal{O}_X))
\,\cong \,\pi^*(f_*(\mathcal{O}_X)) \,\cong\, \pi^*\left(\bigoplus_{l=1}^d \mathcal{O}_{\mathbb{CP}^1}(B_l)\right)
\,=\, \bigoplus_{l=1}^d \mathcal{O}_{\Tw(M)}(B_l)\, ,
\]
where $A_1, \,\cdots,\, A_d,\, B_1,\, \cdots,\, B_d$ are some divisors on $\mathbb{CP}^1$. In view of this, we can 
write the morphism \eqref{pushforward} as
\[
L'(A_1) \oplus \cdots \oplus L'(A_d) \,\longhookrightarrow\, E(B_1) \oplus \cdots\oplus E(B_d)\, ,
\]
which we can think of as a $d \times d$ matrix of morphisms. Since this is a monomorphism, there must be some $1 
\,\le\, j \,\le \,d$, $1 \,\le\, k \,\le\, d$, such that the $(j,\, k)$-th constituent morphism
\[
L'(A_k) \,\longrightarrow\, E(B_j)
\]
is nonzero, which contradicts the irreducibility of $E$.

In this way, both possible cases lead to contradictions. This means that we could not have chosen the morphism $F 
\,:\, E \,\longrightarrow\, E(D)$ in the first place to be anything other than a multiplication by a meromorphic 
function. Hence $E$ must be generically fibrewise simple, and an application of Theorem \ref{thm:result} completes 
the argument.
\end{proof}

\section{The moduli space of fibrewise stable bundles} \label{sect:moduli}

The moduli space of fibrewise stable bundles on the twistor space $\Tw(M)$ of a hyperk\"ahler manifold $M$ can be 
interpreted in terms of rational curves in the twistor space of a certain dual variety $\widehat{M}$. This 
identification of moduli spaces is due to Kaledin and Verbitsky \cite{kaled-verbit}. We present it here, and 
slightly extend it with a technical result which will be used in the next section. Given any complex analytic space 
$X$ with a morphism $X \,\longrightarrow\, \mathbb{CP}^1$, the fibred product of $X \,\longrightarrow\, 
\mathbb{CP}^1$ with the twistor projection $\pi \,:\, \Tw(M) \,\longrightarrow\, \mathbb{CP}^1$ will again be 
denoted by $\Tw(M)_X$, as in the diagram
\[
\xymatrix{\Tw(M)_X \ar[r]^-\varphi \ar[d]_{\pi_X} & \Tw(M) \ar[d]^\pi \\ 
X \ar[r]^f & \mathbb{CP}^1}
\]
By analogy with $\Tw(M)$, we will call a holomorphic bundle on $\Tw(M)_X$ fibrewise stable if all its restrictions 
to the fibres of $\pi_X \,:\, \Tw(M)_X \,\longrightarrow\, X$ are stable. For the rest of this section, we fix a 
topological complex vector bundle $B$ on $M$ whose first two Chern classes $c_1(B)$, $c_2(B)$ are ${\rm 
SU}(2)$-invariant. Recalling that the twistor space $\Tw(M)$ comes equipped with a nonholomorphic projection 
$\sigma\, :\, \Tw(M) \,\longrightarrow\, M$, we can take the pullback bundle $\sigma^*(B)$ on $\Tw(M)$. Any 
holomorphic vector bundle on $\Tw(M)$ that we consider in this section will be assumed to have underlying 
topological structure $\sigma^*(B)$. Similarly, holomorphic bundles on $\Tw(M)_X$ will be assumed to have 
underlying topological structure $\varphi^*\left(\sigma^*(B)\right)$.

Let $I\,\in\, \mathbb{CP}^1$ be any induced complex structure on $M$, and let $\widehat{M}$ denote the moduli space 
of stable holomorphic bundles on $M_I$ with underlying topological structure $B$. Although $\widehat{M}$ need not 
be smooth or reduced, the complex analytic space structure on $\widehat{M}$ induces an almost complex structure on 
the real Zariski tangent spaces of points of $\widehat{M}$, which we denote by $\widehat{I}$.

Now let $J\,\in\,\mathbb{CP}^1$ be any other induced complex structure on $M$. By Theorem \ref{thm:c1c2}, any
stable bundle on $M_I$ with underlying topological structure $B$ is induced by a unique hyperholomorphic
connection on $B$. In turn, such a connection uniquely induces a holomorphic bundle on $M_J$, which
is stable since hyperholomorphic connections are Yang-Mills (Theorem 2.3 in \cite{verbit1}). It follows from
this that the underlying set of the moduli space of stable bundles on $M_J$ can be identified
with $\widehat{M}$, and we denote the corresponding almost complex structure on $\widehat{M}$
by $\widehat{J}$. In \cite{verbit1}, it is shown that if $(I,\, J,\, K\, =\, IJ)$ is a quaternionic triple on $M$,
\[
\left(\widehat{I},\, \widehat{J},\, \widehat{K} \,= \,\widehat{IJ} \,= \,\widehat{I}\widehat{J}\right)
\]
will be a quaternionic triple on $\widehat{M}$, and moreover there exists a metric on the real Zariski tangent 
space of $\widehat{M}$, compatible with this quaternionic structure, which gives $\widehat{M}$ the structure of a 
\emph{singular hyperk\"ahler variety}; see \cite{verbit1} for the precise definition and proof. Following 
\cite{kaled-verbit}, we will call $\widehat{M}$ the \emph{Mukai dual} of $M$. In case $M$ is a hyperk\"ahler 
surface, it follows from the work of Mukai \cite{mukai} that $\widehat{M}$ is actually smooth.

Although in general the Mukai dual $\widehat{M}$ is singular (and noncompact), one can still construct its twistor 
space, in the same way that we did for $M$ in Section \ref{sect:preliminaries}. The twistor space 
$\Tw(\widehat{M})$ is a complex analytic space parametrizing the induced complex structures at points of 
$\widehat{M}$; it is singular if $\widehat{M}$ is. To ease notation, in the rest of this section we will denote the 
Mukai dual twistor space $\Tw(\widehat{M})$ by $\widehat{Z}$. Just like the usual twistor space, $\widehat{Z}$ 
comes equipped with a holomorphic twistor projection
\[
\widehat{\pi} \,:\, \widehat{Z} \,\longrightarrow \,\mathbb{CP}^1,
\]
whose holomorphic sections will be called twistor lines. The set $\Sec(\widehat{\pi})$ of such twistor lines has the structure of a complex analytic space as a subset of the Douady space of rational curves in $\widehat{Z}$.

Now let $\mathcal{M}^s_{\mathrm{fib}}$ denote the set of fibrewise stable bundles on the original twistor space $\Tw(M)$. Since any such bundle is irreducible by Theorem \ref{thm:result}, it is in particular stable, so we have a set-theoretic inclusion
\[
\mathcal{M}^s_{\mathrm{fib}}\,\subset\, \mathcal{M}^s\, ,
\]
where $\mathcal{M}^s$ denotes the moduli space of stable holomorphic bundles on $\Tw(M)$. In fact, since stability 
is an open property, we know that $\mathcal{M}^s_{\mathrm{fib}}$ is an open subset of $\mathcal{M}^s$, and thus it 
inherits from $\mathcal{M}^s$ the structure of a complex analytic space. Let $E$ be any element of 
$\mathcal{M}^s_{\mathrm{fib}}$. From the discussion above, for any $I \in \mathbb{CP}^1$, the moduli space of 
stable bundles on $M_I$ has underlying set $\widehat{M}$. In this way, $E$ defines a (set-theoretic) map
\[
\begin{array}{ccc}
\mathbb{CP}^1 & \longrightarrow & \widehat{M} \\
I & \longmapsto & E_I,
\end{array}
\]
where $E_I$ is the restriction of $E$ to the fibre $\pi^{-1}(I)\,=\,M_I$, and this map in turn defines a (set-theoretic)
section of the Mukai dual twistor projection $\widehat{\pi}\,:\, \widehat{Z}\,\longrightarrow\,
\mathbb{CP}^1$. For example, it's not hard to see that if $E$ is hyperholomorphic, the resulting
section is just a horizontal twistor line. In general, the section of $\widehat{\pi}$ induced by
$E$ will be holomorphic, and will thus be a twistor line. Furthermore, any twistor line can be obtained in this way from
a unique $E$, as the next result shows.

\begin{theorem}\label{thm:fib-sec}
Any fibrewise stable bundle on $\Tw(M)$ induces in a natural way a twistor line in the Mukai dual twistor space $\widehat{Z} = \Tw(\widehat{M})$, and the resulting map
\[
\mathcal{M}^s_{\mathrm{fib}} \,\stackrel{\cong}{\longrightarrow} \,\Sec(\widehat{\pi})
\]
is an isomorphism of complex analytic spaces. Moreover, there exists an open cover $\left\{U_\alpha\right\}$ of $\widehat{Z}$, together with holomorphic vector bundles $\mathcal{E}_\alpha$ on the corresponding open neighborhoods $\pi_{\widehat{Z}}^{-1}(U_\alpha) \subseteq \Tw(M)_{\widehat{Z}}$, with the following property: for any complex analytic space $X$, morphism $X \,\longrightarrow\,
\mathbb{CP}^1$ and fibrewise stable bundle $\mathcal{F}$ on $\Tw(M)_X$, there exists a unique map $g \,:\, X \,\longrightarrow\,
\widehat{Z}$ over $\mathbb{CP}^1$, such that, for every index $\alpha$,
\[
\mathcal{F} \,\cong\, \psi^* \mathcal{E}_\alpha\ \textrm{ over }\ \pi_X^{-1}(g^{-1}(U_\alpha))\, ,
\]
where the map $\psi \,:\, \Tw(M)_X \,\longrightarrow\, \Tw(M)_{\widehat{Z}}$ is induced by $g$, as in the diagram
\begin{equation} \label{fib-sec-univ}
\xymatrix{\Tw(M)_X \ar@{-->}[r]^-\psi \ar[d]_{\pi_X} & \Tw(M)_{\widehat{Z}} \ar[r] \ar[d]^{\pi_{\widehat{Z}}} & \Tw(M) \ar[d]^\pi \\ X \ar@{-->}[r]^-g & \widehat{Z} \ar[r]^{\widehat{\pi}} & \mathbb{CP}^1}
\end{equation}
\end{theorem}

\begin{proof}
See Section 7 of \cite{kaled-verbit}.
\end{proof}

In other words, $\widehat{Z}$ is a coarse moduli space of fibrewise stable bundles on $\Tw(M)$ which locally admits 
a universal family $\mathcal{E}_\alpha$, but these local universal families $\left\{\mathcal{E}_\alpha\right\}$ 
need not come from a single global universal family.

Now look at the special case that $f \,:\, X \,\longrightarrow\,
\mathbb{CP}^1$ is a branched cover of $\mathbb{CP}^1$ by a smooth curve $X$. By Theorem \ref{thm:fib-sec} above, a fibrewise stable bundle on the fibred product $\Tw(M)_X$ gives rise to a \emph{multisection} of the Mukai dual twistor projection $\widehat{\pi} \,:\, \widehat{Z}\,\longrightarrow\,
\mathbb{CP}^1$ over $X$, i.e., a morphism
\[
\xymatrix{ & \widehat{Z} \ar[d]^{\widehat{\pi}} \\ X \ar@{-->}[ur] \ar[r]_f & \mathbb{CP}^1 }
\]
We thus get a map
\begin{equation} \label{multi-fib-sec}
{\mathcal{M}^s_{\mathrm{fib}}}_X \longrightarrow \Sec_X(\widehat{\pi}),
\end{equation}
where ${\mathcal{M}^s_{\mathrm{fib}}}_X$ denotes the moduli space of fibrewise stable bundles on $\Tw(M)_X$, and $\Sec_X(\widehat{\pi})$ the space of multisections of $\widehat{\pi}$ over $X$. In contrast to the result of Theorem \ref{thm:fib-sec}, this map need not be injective. However, it is surjective, as the following lemma shows.

\begin{lemma}\label{thm:multi-fib}
Let $g \,:\, X \,\longrightarrow\,
\widehat{Z}$ be a multisection of the Mukai dual twistor projection $\widehat{\pi} \,:\,
\widehat{Z} \,\longrightarrow\,
\mathbb{CP}^1$. There exists a fibrewise stable bundle on $\Tw(M)_X$ which gets mapped to $g$ via the map \eqref{multi-fib-sec}.
\end{lemma}

\begin{proof}
The argument is the same as in the proof of Lemma 7.5 in \cite{kaled-verbit}. The map $g \,:\, X \,\longrightarrow\, \widehat{Z}$ induces a map $\psi \,:\, \Tw(M)_X
\,\longrightarrow\, \Tw(M)_{\widehat{Z}}$ of fibred products, as in the diagram \eqref{fib-sec-univ}. From Theorem \ref{thm:fib-sec}, we know that there are open sets $U_\alpha \subseteq \widehat{Z}$ covering $\widehat{Z}$, and universal bundles $\mathcal{E}_\alpha$ on the corresponding open neighborhoods $\pi_{\widehat{Z}}^{-1}\left(U_\alpha\right) \subseteq \Tw(M)_{\widehat{Z}}$. Passing from $\widehat{Z}$ to $X$, we have the open sets $g^{-1}(U_\alpha) \subseteq X$, and the pullbacks of the bundles $\mathcal{E}_\alpha$ by $\psi$ on the corresponding open neighborhoods $\pi_X^{-1}\left(g^{-1}(U_\alpha)\right) \subseteq \Tw(M)_X$. For simplicity, we will denote the preimage neighborhood $g^{-1}(U_\alpha)$ by $U_\alpha$ as well, and the pullback of the bundle $\mathcal{E}_\alpha$ again by $\mathcal{E}_\alpha$. If we can somehow glue these $\mathcal{E}_\alpha$ into a bundle on $\Tw(M)_X$, it's clear that it will be fibrewise stable, and that it will be mapped to $g$ via the map \eqref{multi-fib-sec}.

By compactness of $X$, we can choose a finite sub-cover $U_1, \,\cdots,\, U_n$ of $\left\{U_\alpha\right\}$. The corresponding bundles $\mathcal{E}_1, \,\cdots,\, \mathcal{E}_n$ are isomorphic on overlaps by their universal property. In other words, for any $1 \le i, j \le n$, we have isomorphisms
\[
h_{ij} : \left.\mathcal{E}_j\right|_{\pi_X^{-1}(U_i \cap U_j)} \stackrel{\cong}{\longrightarrow} \left.\mathcal{E}_i\right|_{\pi_X^{-1}(U_i \cap U_j)}.
\]
For any $i, j, k$, the composition
\[
h_{ij} \circ h_{jk} \circ h_{ki} \,:\,
\left.\mathcal{E}_i\right|_{\pi_X^{-1}(U_i \cap U_j \cap U_k)}
\,\stackrel{\cong}{\longrightarrow} \,\left.\mathcal{E}_i\right|_{\pi_X^{-1}(U_i \cap U_j \cap U_k)}
\]
need not equal the identity map. However, for any point $P \in U_i \cap U_j \cap U_k$, the restriction $\left.\mathcal{E}_i\right|_{\pi_X^{-1}(P)}$ is a stable bundle, hence in particular simple, i.e.,
\[
\Hom\left(\left.\mathcal{E}_i\right|_{\pi_X^{-1}(P)}, \left.\mathcal{E}_i\right|_{\pi_X^{-1}(P)}\right) = \mathbb{C}.
\]
It follows from this that we have
\[
h_{ij} \circ h_{jk} \circ h_{ki} \,=\, \theta_{ijk} \textrm{Id}_{\mathcal{E}_i}, \ 
\theta_{ijk} \,\in\, \mathcal{O}_X^*(U_i \cap U_j \cap U_k).
\]
The collection $\left\{\theta_{ijk}\right\}$ is a \v{C}ech 2-cocycle defining an element of the cohomology group ${{}H}^2(X,\, \mathcal{O}^*_X)$. Thus the bundles $\mathcal{E}_i$ together with the isomorphisms $h_{ij}$ define a \emph{twisted sheaf} on $\Tw(M)_X$ in the sense of C\u{a}ld\u{a}raru \cite{caldararu}, and it's not hard to verify that the $\mathcal{E}_i$ glue into an actual sheaf if and only if the element of ${{}H}^2(X, \mathcal{O}^*_X)$ defined by the collection $\left\{\theta_{ijk}\right\}$ is zero. But since $X$ is a curve, ${{}H}^2(X, \mathcal{O}^*_X) = 0$, as can be seen from the following portion of the long exact cohomology sequence of the exponential sheaf sequence of $X$:
\[
\ldots \,\longrightarrow\, {{}H}^2(X, \,\mathcal{O}_X) \,\longrightarrow\,
{{}H}^2(X,\, \mathcal{O}^*_X) \,\longrightarrow\, {{}H}^3(X, \,\mathbb{Z}) \,\longrightarrow\, \ldots
\]
This completes the proof.
\end{proof}

\section{Irreducible bundles on $\mathrm{Tw}(M)$ of composite rank} \label{sect:composite}

We note that Theorem \ref{thm:prime} gives some hope that the full converse of Theorem \ref{thm:result} might be 
true, in other words, that an arbitrary irreducible bundle $E$ on the twistor space $\Tw(M)$ of a compact simply 
connected hyperk\"ahler manifold $M$ is generically fibrewise stable. This, however, turns out not to be true, and 
in this section we will construct examples of irreducible but nowhere fibrewise stable bundles on $\Tw(M)$ of any 
composite rank. We will carry out the construction on the twistor space of a K3 surface.

Recall that a \emph{K3 surface} is a compact simply connected complex surface $M$ with trivial canonical bundle. A 
nonzero holomorphic section of the canonical bundle of $M$ is a holomorphic symplectic form, making it into a K\"ahler 
holomorphic symplectic manifold. As a consequence of Yau's theorem proving Calabi's conjecture, \cite{yau}, the
manifold $M$ admits a hyperk\"ahler structure.

The main result of this section follows.

\begin{theorem}\label{thm:composite}
Let $M$ be a K3 surface. There exist examples of irreducible but nowhere fibrewise stable bundles of any composite 
rank on its twistor space $\Tw(M)$.
\end{theorem}

Before going ahead with the proof, we give a concise overview of the argument. The construction will be carried out 
in the following three steps:

\begin{enumerate}
\item[1.] Given any composite number, we write it as a product $dr$, where $d$ is prime. We choose a topological 
complex vector bundle $B$ on $M$ of rank $r$ that admits stable structures in every induced complex structure of 
$M$, so that the corresponding Mukai dual variety $\widehat{M}$ is nonempty.

\item[2.] We choose a branched cover of $\mathbb{CP}^1$ by a smooth curve $X$, $f \,:\, X \,\longrightarrow\, 
\mathbb{CP}^1$, in such a way that the Mukai dual twistor projection $\widehat{\pi} \,:\, \widehat{Z} \,=\, 
\Tw(\widehat{M}) \,\longrightarrow\,
 \mathbb{CP}^1$ admits a multisection over $X$ which does not come from a twistor line in $\widehat{Z}$. Applying Lemma \ref{thm:multi-fib}, this gives rise to a fibrewise stable bundle $F$ on the fibred product $\Tw(M)_X$, as in the diagram
\[
\xymatrix{\Tw(M)_X \ar[r]^-\varphi \ar[d]_{\pi_X} & \Tw(M) \ar[d]^\pi \\ X \ar[r]^f & \mathbb{CP}^1}
\]
The bundle $F$ will have the property that for generic $I\,\in\, \mathbb{CP}^1$, the restrictions of $F$ to the fibres 
of $\pi_X$ corresponding to distinct elements in $f^{-1}(I)$ are nonisomorphic as bundles on $M_I$.

\item[3.] Taking the pushforward $E \,:= \,\varphi_* F$ by the map $\varphi$ in the diagram above, we show that $E$ 
is an irreducible bundle on $\Tw(M)$ of rank $dr$ which is nowhere fibrewise stable.
\end{enumerate}

\begin{proof}[Proof of Theorem \ref{thm:composite}]
Take any composite number, and write it as a product $dr$, where $d$ is some prime number. We fix once and for the rest of the proof a branched cover $f \,:\, X \,\longrightarrow\,
 \mathbb{CP}^1$ as follows. Let $X \,=\, \mathbb{CP}^1$, and choose any local coordinate $z$ about any point in $\mathbb{CP}^1$. The map $f \,:\, X \,=\, \mathbb{CP}^1 \,\longrightarrow\,
 \mathbb{CP}^1$ is given by $f(z) \,= \,z^d$.

\subsection*{Step 1} We first choose a topological complex vector bundle on $M$ of rank $r$ that admits stable 
holomorphic structures in every induced complex structure of $M$. Fix $I\,\in\, \mathbb{CP}^1$. Using Serre's 
construction, one can show that there exists a stable bundle of rank $r$ on $M_I$ with first Chern class zero (see 
Theorem 5.1.6 in \cite{huyb-lehn}). By Theorem \ref{thm:c1c2}, we know that such bundle comes from a unique 
hyperholomorphic connection $\nabla$ on its underlying topological bundle $B$. Since hyperholomorphic connections 
are Yang-Mills (Theorem 2.3 in \cite{verbit1}), $\nabla$ gives rise to stable holomorphic structures on $B$ in all 
other induced complex structures of $M$. It follows from this that the Mukai dual variety $\widehat{M}$ associated 
to $B$ is nonempty. As was mentioned in the previous section, $\widehat{M}$ is smooth since $M$ is a surface, and 
is thus a (noncompact) hyperk\"ahler manifold. As in the previous section, let $\widehat{Z}\,=\, \Tw(\widehat{M})$ 
denote the twistor space of $\widehat{M}$, and $\widehat{\pi} \,:\, \widehat{Z} \,\longrightarrow\, \mathbb{CP}^1$ 
its holomorphic twistor projection.

\subsection*{Step 2} Recall that sections of the twistor projection $\widehat{\pi}\, :\, \widehat{Z}
\,\longrightarrow\,
 \mathbb{CP}^1$ are called twistor lines, and the set of twistor lines in $\widehat{Z}$ is denoted by
$\Sec(\widehat{\pi})$. We will also be interested in the multisections of $\widehat{\pi} \,:\,
\widehat{Z}\,\longrightarrow\,
\mathbb{CP}^1$ over $f \,:\, X\,\longrightarrow\,
\mathbb{CP}^1$, and the set of such multisections will be denoted by $\Sec_X(\widehat{\pi})$. Viewed as Douady
spaces of morphisms, both $\Sec(\widehat{\pi})$ and $\Sec_X(\widehat{\pi})$ have a complex analytic structure, and
composition with $f \,:\, X\,\longrightarrow\, \mathbb{CP}^1$ induces an analytic map
\begin{equation} \label{secheniya}
\begin{array}{ccc}
\Sec(\widehat{\pi}) & \longrightarrow & \Sec_X(\widehat{\pi}) \\
\left[s\right] & \longmapsto & \left[s \circ f\right].
\end{array}
\end{equation}
We would like to show that this map is not surjective. We do this by examining the induced map of Zariski tangent spaces, which represent infinitesimal deformations of morphisms.

Fix a twistor line $s \,:\, \mathbb{CP}^1\,\longrightarrow\, \widehat{Z}$ (e.g. a horizontal twistor line). We have the following natural short exact sequence of holomorphic bundles on $\widehat{Z}$:
\[ 
0\, \longrightarrow\, {{}T}^{1,0}_{\widehat{\pi}}\, \longrightarrow\, {{}T}^{1,0}{ \widehat{Z}} \, \stackrel{d\widehat{\pi}}{\longrightarrow}\,
\widehat{\pi}^*{{}T}^{1,0}{{\mathbb C}{\mathbb P}^1} \, \longrightarrow\, 0\, .
\]
Here ${{}T}^{1,0}_{\widehat{\pi}}$ is the relative holomorphic tangent bundle for the projection
$\widehat{\pi}$, and $d\widehat{\pi}$ is the differential of $\widehat{\pi}$. Pulling this sequence back to
$\mathbb{CP}^1$ via $s \,:\, \mathbb{CP}^1\,\longrightarrow\, \widehat{Z}$, we get:
\begin{equation} \label{relative}
0\, \longrightarrow\, s^*{{}T}^{1,0}_{\widehat{\pi}}\, \longrightarrow\, s^*{{}T}^{1,0}{ \widehat{Z}} \, \stackrel{s^*d\widehat{\pi}}{\longrightarrow}\,
{{}T}^{1,0}{{\mathbb C}{\mathbb P}^1} \, \longrightarrow\, 0\, .
\end{equation}
The Zariski tangent space of $\Sec(\widehat{\pi})$ at $\left[s\right]$ can be identified with the space of global sections of the bundle $s^*{{}T}^{1,0}_{\widehat{\pi}}$,
\[
T_{\left[s\right]} \Sec(\widehat{\pi}) \,\cong\, {{}H}^0\left(\mathbb{CP}^1,\, s^*{{}T}^{1,0}_{\widehat{\pi}}\right)
\]
(see Section 2.3 in \cite{debarre}), and similarly,
\[
T_{\left[s \circ f\right]} \Sec_X(\widehat{\pi}) \,\cong\, {{}H}^0\left(X,\, f^*\left(s^*{{}T}^{1,0}_{\widehat{\pi}}\right)\right).
\]
We now describe the structure of the vector bundle $s^*{{}T}^{1,0}_{\widehat{\pi}}$ on $\mathbb{CP}^1$. Since $T\mathbb{CP}^1 \cong \mathcal{O}_{\mathbb{CP}^1}(2)$ and the normal bundle of the twistor line $s \,: \,\mathbb{CP}^1 \,\longrightarrow\,
 \widehat{Z}$ is isomorphic to $\mathcal{O}_{\mathbb{CP}^1}(1)^{\oplus n}$, where $n$ is the complex dimension of $\widehat{M}$ (see Theorem 1(2) in \cite{hitchin2}), the sequence \eqref{relative} has the form
\[
0\, \longrightarrow\, s^*{{}T}^{1,0}_{\widehat{\pi}}\, \longrightarrow\, \mathcal{O}_{\mathbb{CP}^1}(2) \oplus \mathcal{O}_{\mathbb{CP}^1}(1)^{\oplus n} \, \longrightarrow\,
\mathcal{O}_{\mathbb{CP}^1}(2) \, \longrightarrow\, 0\, .
\]
Here the $\mathcal{O}_{\mathbb{CP}^1}(2)$ term in the middle gets mapped identically to $\mathcal{O}_{\mathbb{CP}^1}(2)$ on the right. It follows from this that $s^*{{}T}^{1,0}_{\widehat{\pi}}
\,\cong\, \mathcal{O}_{\mathbb{CP}^1}(1)^{\oplus n}$, and recalling that $X \,=\, \mathbb{CP}^1$ and that $f \,:\, X
\,\longrightarrow\,
 \mathbb{CP}^1$ is a degree $d$ map, we have $f^*\left(s^*{{}T}^{1,0}_{\widehat{\pi}}\right)
\,\cong\, \mathcal{O}_{\mathbb{CP}^1}(d)^{\oplus n}$. In view of this, the map of Zariski tangent spaces
\[
T_{\left[s\right]} \Sec(\widehat{\pi})\,\longrightarrow\, T_{\left[s \circ f\right]} \Sec_X(\widehat{\pi})
\]
induced by the morphism \eqref{secheniya} takes the form
\[
{{}H}^0\left(\mathbb{CP}^1,\, \mathcal{O}_{\mathbb{CP}^1}(1)^{\oplus n}\right) \,\longrightarrow
\, {{}H}^0\left(X, \,\mathcal{O}_{\mathbb{CP}^1}(d)^{\oplus n}\right)\, .
\]
Since $d \,>\, 1$, we see that this map cannot be surjective.

It follows from this that the multisection $s \circ f \,:\, X\,\longrightarrow\,
 \widehat{Z}$ can be deformed to a multisection $g \,:\, X \,\longrightarrow\,
 \widehat{Z}$, as in the diagram
\[
\xymatrix{ & \widehat{Z} \ar[d]^{\widehat{\pi}} \\ X \ar[ur]^g \ar[r]_f & \mathbb{CP}^1, }
\]
that does not factor through $f$, that is, does not come from from a twistor line in $\widehat{Z}$. Since the degree of $f$ is a prime number $d$, it follows that the map $g$ must be injective. Applying Lemma \ref{thm:multi-fib} to $g$, we get a fibrewise stable bundle $F$ of rank $r$ on the fibred product $\Tw(M)_X$, as in the diagram
\begin{equation} \label{diagramma3}
\xymatrix{\Tw(M)_X \ar[r]^-\varphi \ar[d]_{\pi_X} & \Tw(M) \ar[d]^\pi \\ X \ar[r]^f & \mathbb{CP}^1.}
\end{equation}
By choice of $g$, the bundle $F$ will have the property that for generic $I \in \mathbb{CP}^1$, the restrictions of $F$ to the fibres of $\pi_X$ corresponding to distinct elements in $f^{-1}(I)$ are nonisomorphic stable bundles on $M_I$ of degree zero.

\subsection*{Step 3} Let $E \,=\, \varphi_*F$ be the pushforward of $F$ by $\varphi$ in the diagram 
\eqref{diagramma3}. First, observe that $E$ is locally free. Being a local statement on $\Tw(M)$, this follows from 
the fact that $\varphi_* \mathcal{O}_{\Tw(M)_X}$ is locally free, which itself follows from the fact that $\varphi$ 
is proper and finite. Thus, $E$ is a vector bundle, and it has rank $dr$, since the degree of the map $f \,:\, 
X\,\longrightarrow\,\mathbb{CP}^1$ is $d$, and the rank of $F$ is $r$. Second, the vector bundle $E$ on $\Tw(M)$ 
is nowhere fibrewise stable. Indeed, for any $I\,\in\,\mathbb{CP}^1$ outside the branch locus of $f$, we know from 
the construction of $F$ that the restriction $E_I \,=\, \left.E\right|_{\pi^{-1}(I)}$ decomposes as
\begin{equation} \label{razlozh}
E_I \,\cong\, E_1 \oplus \cdots\oplus E_d\, ,
\end{equation}
where the $E_i$ correspond to restrictions of $F$ over points in $f^{-1}(I)$, and are (nonisomorphic) stable bundles on $M_I$ of rank $r$ and degree zero. This shows that $E_I$ is non-stable for all $I$ outside a finite subset, and since non-stability is a closed property, it follows that the same must be true for all $I \in \mathbb{CP}^1$. We emphasize that, unlike $E_I$, the bundles $E_i$ in the decomposition \eqref{razlozh} live only on $M_I$ and cannot be extended to the whole $\Tw(M)$.

It remains to show that $E$ is irreducible as a bundle on $\Tw(M)$. Let $G\,\subset\, E$ be any subsheaf of lower 
rank; we can assume that $G$ is reflexive. Notice that the restriction $G_I$ of $G$ to any fibre $\pi^{-1}(I) 
\,\subset\, \Tw(M)$ has degree zero. This is certainly true for generic $I$ in the sense of Definition 
\ref{def:generic}, hence by continuity it is true for all $I\,\in\, \mathbb{CP}^1$. Recalling that $\dim_{\mathbb{C}} 
\Tw(M) \,=\, 3$ since $M$ is a surface, and that for reflexive sheaves the codimension of the
singular locus is at least three 
(Lemma 1.1.10 in Chapter 2 of \cite{okonek}), we see that $G$ is a vector bundle outside a finite subset of 
$\Tw(M)$. Let
\[
\Delta \,=\, \textrm{ Branch locus of } f\ \cup \textrm{ Singularity set of } G.
\]
Fix $I\,\in\, \mathbb{CP}^1 \setminus \Delta$. The restriction of the sheaf inclusion $G\,\subset
\,E$ to the fibre $\pi^{-1}(I)\, =\, M_I$ is a sheaf monomorphism
\begin{equation} \label{vlozh}
G_I \,\longhookrightarrow\, E_1 \oplus \cdots \oplus E_d\, ,
\end{equation}
where we have used the decomposition \eqref{razlozh}. Notice that both $G_I$ and the $E_i$ are vector bundles here.

We claim that there exists a choice of a subset $\left\{i_1,\, \cdots,\, i_t \right\}
\,\subseteq\, \left\{1,\, \cdots,\, d\right\}$ such that the composition of \eqref{vlozh} with the corresponding projection,
\[
G_I \,\longrightarrow\, E_1 \oplus \cdots \oplus E_d\, \longrightarrow \,E_{i_1} \oplus \cdots \oplus E_{i_t}\, ,
\]
is generically an isomorphism.

To prove the above claim, let $1\,\le\,j_1\,\le\,d$ be arbitrary, and look at the composition
\[
G_I \longrightarrow E_1 \oplus \cdots \oplus E_d \longrightarrow \bigoplus_{l \ne j_1} E_l
\,=\, E_1 \oplus \cdots \oplus \widehat{E_{j_1}} \oplus \cdots \oplus E_d.
\]
Let $K_{j_1}$ denote the kernel of this composition. We have the following diagram with exact rows:
\[
\xymatrix{0 \ar[r] & *+<1pc>{K_{j_1}} \ar[r] \ar@{-->}[d] & *+<1pc>{G_I} \ar[r] \ar@{^{(}->}[d]^{\gamma} & *+<1pc>{\bigoplus_{l \ne j_1} E_l} \ar@{=}[d] & \\ 0 \ar[r] & E_{j_1} \ar[r] & E_1 \oplus \cdots \oplus E_d \ar[r] & \bigoplus_{l \ne j_1} E_l \ar[r] &0}
\]
Just like $G_I$, the sheaf $K_{j_1}$ has degree zero. Indeed, if we had $\deg K_{j_1}\,>\, 0$, then the
composition $K_{j_1} \,\subset\, G_I \,\subset\, E_1\bigoplus\cdots\bigoplus E_d$ would destabilize
the polystable vector bundle $E_1 \,\bigoplus\, \cdots\, \bigoplus\, E_d$, while if $\deg K_{j_1} \,<\, 0$, the fact
that $\deg G_I\,=\, 0$ would imply that the image of $G_I$ under the rightmost map in the first row of the
above diagram would have positive degree, thus destabilizing $\bigoplus_{l \ne j_1} E_l$. Since $E_{j_1}$ is
stable and also has degree zero, the condition $\deg K_{j_1} \,=\, 0$ implies that the left-most vertical arrow
in the above diagram is either zero or generically an isomorphism. If the latter is true for every $j_1$ from 1 to
$d$, then $\rk G_I\,=\, \rk \left[E_1 \bigoplus \cdots \bigoplus E_d\right]$, but this cannot be as $G$
was chosen to be a subsheaf of $E$ on $\Tw(M)$ of lower rank. Fixing an index $j_1$ such that
$K_{j_1} \,=\, 0$, the composition
\[
G_I \,\longrightarrow\, E_1 \oplus \cdots \oplus E_d \,\longrightarrow\, \bigoplus_{l \ne j_1} E_l
\]
must be a monomorphism. If $\rk G_I\,=\,\rk \bigoplus_{l \ne j_1} E_l$, we stop here. If not, we repeat the argument above with $\left\{1, \,\cdots,\, d\right\}$ replaced by $\left\{1,\, \cdots,\, \widehat{j_1},\, \cdots,\, d\right\}$ to conclude the existence of an index $j_2 \in \left\{1, \,\cdots,\, \widehat{j_1}, \,\cdots,\, d\right\}$ such that the composition
\[
G_I \,\longrightarrow\, \bigoplus_{l \ne j_1} E_l \,\longrightarrow \,\bigoplus_{l \ne j_1, j_2} E_l
\]
is still a monomorphism. At a certain point, after having chosen some indices $j_1,\, j_2,\, \cdots,\, j_s$ in
this manner, and letting $i_1 , \,i_2,\, \cdots,\, i_t$ denote the other indices, we will arrive at a
monomorphism $G_I \,\longhookrightarrow\, E_{i_1} \bigoplus \cdots \bigoplus E_{i_t}$ with $\rk G_I \,=\,
\rk \left[E_{i_1} \bigoplus \cdots \bigoplus E_{i_t}\right]$.

If $t \,=\, 0$, then $G_I \,=\, 0$, and so $G\, =\, 0$. If this happens for an arbitrary choice of a subsheaf $G \subset 
E$, then clearly $E$ must be irreducible. Assume for contradiction that $t\,>\, 0$.

We thus have that for some subset $\left\{i_1,\, \cdots,\, i_t \right\} \,\subseteq \,
\left\{1,\, \cdots,\, d\right\}$, the map
\begin{equation} \label{vlozh-proj}
G_I\, \longrightarrow\, E_{i_1} \oplus \cdots \oplus E_{i_t}
\end{equation}
obtained as the composition of \eqref{vlozh} with the corresponding projection is a monomorphism
of sheaves with the property that $\rk G_I\, =\, \rk \left[E_{i_1} \bigoplus \cdots \bigoplus E_{i_t}\right]$. But
since also $\deg G_I \,=\, \deg \left[E_{i_1} \bigoplus \cdots \bigoplus E_{i_t}\right]\,=\, 0$, it must be that the
corresponding map of line bundles
\[
\det G_I\,\longrightarrow\, \det \left(E_{i_1} \oplus \cdots \oplus E_{i_t}\right)
\]
is an isomorphism. Recalling that both $G_I$ and $E_{i_1}\bigoplus\cdots\bigoplus E_{i_t}$ are vector bundles, it
follows that \eqref{vlozh-proj} must also be an epimorphism, and hence an isomorphism.
This proves the claim.

Identifying $G_I$ with $E_{i_1}
\bigoplus\cdots\bigoplus E_{i_t}$ on $M_I$ in the morphism \eqref{vlozh}, we also have that for any $j$ in $\left\{1, \,
\cdots,\, d\right\} \setminus \left\{i_1,\, \cdots,\, i_t\right\}$, the composition
\[
G_I\,\cong\, E_{i_1} \oplus \cdots \oplus E_{i_t}\,\longrightarrow\, E_1 \oplus \cdots \oplus E_d\,\longrightarrow\,E_j
\]
is zero. This follows from the fact that by construction, the bundles $E_1, \,\cdots,\, E_d$ are all stable
on $M_I$ of the same rank and degree, and are pairwise nonisomorphic.

It follows from all this that the choice of a subset $\left\{i_1,\,\cdots,\, i_t \right\}\,\subseteq\,
\left\{1,\, \cdots, \,d\right\}$ as described above is uniquely determined by the morphism \eqref{vlozh}. Moreover, it's not hard to see 
that in a neighborhood $U$ of $I$ in $\mathbb{CP}^1 \setminus \Delta$ which is evenly covered by the map $f \,:\, X 
\,\longrightarrow\,\mathbb{CP}^1$, carrying out the same procedure for every other point in $U$ yields the same 
choice of subset of $\left\{1,\,\cdots,\, d\right\}$. Thus, a nonzero subsheaf $G \subset E$ gives a consistent choice 
of $t$ sheets of the covering $f \,:\, X \,\longrightarrow\, \mathbb{CP}^1$ for every point in the nonempty Zariski 
open set $\mathbb{CP}^1 \setminus \Delta$, which contradicts the fact that $X\,=\, \mathbb{CP}^1$ is connected. We 
must have that the only subsheaf of $E$ of lower rank on $\Tw(M)$ is the zero subsheaf, i.e., $E$ is irreducible.
\end{proof}

\section*{Acknowledgements}

The authors would like to thank Ajneet Dhillon, Jacques Hurtubise and Misha Verbitsky for the many valuable 
discussions and suggestions. The study of the second author has been funded within the framework of the HSE 
University Basic Research Program and the Russian Academic Excellence Project '5-100'. The first author is 
supported by a J. C. Bose Fellowship.

\end{document}